\documentclass{article}
\usepackage[T1]{fontenc}
\usepackage[utf8]{inputenc}
\usepackage[american,shorthands=off]{babel}
\usepackage{amsmath,amssymb,amsthm,mathtools}
\usepackage[shortlabels]{enumitem}
\usepackage{graphicx}
\usepackage[noadjust]{cite}
\usepackage{xurl}
\usepackage{etoolbox}
\usepackage[title]{appendix}
\usepackage[hidelinks]{hyperref}

\date{}

\makeatletter
\newcount\manuscript@total
\newcommand*\manuscript@declarefield[1]{%
  \csdef{#1}##1{\csgdef{manuscript@\manuscript@id @#1}{##1}}}
\forcsvlist{\manuscript@declarefield}{givenname,surname,zblid,mrid,orcid,
  department,organisation,rorid,address,zip,city,country,affemail}
\newcommand*\manuscriptauthor[3]{%
  \begingroup\def\manuscript@id{#1}\setbox\@tempboxa\hbox{#2}\endgroup
  \csgdef{manuscript@#1@short}{#3}%
  \global\manuscript@total=#1\relax}
\newcommand*\manuscriptaffiliation[2]{%
  \begingroup\def\manuscript@id{#1}\setbox\@tempboxa\hbox{#2}\endgroup}
\newcommand*\manuscript@field[2]{\csuse{manuscript@#1@#2}}
\newcommand*\manuscript@name[1]{%
  \manuscript@field{#1}{givenname}\space\manuscript@field{#1}{surname}}
\newcommand*\titlemark[1]{\gdef\manuscript@titlemark{#1}}
\newcommand*\classification[2][]{%
  \gdef\manuscript@classification{#2~(primary); #1~(secondary)}}
\newcommand*\keywords[1]{\gdef\manuscript@keywords{#1}}

\newcommand*\preparemanuscriptauthors{%
  \author{\manuscript@name{1}, \manuscript@name{2},\\
    \manuscript@name{3}, \manuscript@name{4}\\
    and \manuscript@name{5}}}
\newcommand*\printmanuscriptmetadata{%
  \par\noindent\textit{Mathematics Subject Classification 2020:}\enskip
  \manuscript@classification.\par
  \noindent\textit{Keywords:}\enskip\manuscript@keywords.\par}

\newenvironment{ack}{\paragraph*{Acknowledgments.}}{}
\newenvironment{funding}{\paragraph*{Funding.}}{}

\newcommand*\manuscript@printaffiliation[1]{%
  \par\medskip\noindent{\bfseries\manuscript@name{#1}}\par\nobreak
  \manuscript@field{#1}{department}, \manuscript@field{#1}{organisation}\\*
  \ifcsempty{manuscript@#1@address}{}{\manuscript@field{#1}{address}, }%
  \ifnum\pdfstrcmp{\manuscript@field{#1}{country}}{Canada}=0
    \manuscript@field{#1}{city}, \manuscript@field{#1}{zip}%
  \else
    \manuscript@field{#1}{zip}~\manuscript@field{#1}{city}%
  \fi, \manuscript@field{#1}{country};\\*
  \href{mailto:\manuscript@field{#1}{affemail}}{\mbox{\manuscript@field{#1}{affemail}}}%
  \par\noindent Author IDs:\space
  zbMATH \href{https://zbmath.org/authors/\manuscript@field{#1}{zblid}}{\manuscript@field{#1}{zblid}}%
  \ MR \href{https://mathscinet.ams.org/mathscinet/MRAuthorID/\manuscript@field{#1}{mrid}}{\manuscript@field{#1}{mrid}}%
  \ ORCID \href{https://orcid.org/\manuscript@field{#1}{orcid}}{\manuscript@field{#1}{orcid}}}
\AtEndDocument{\par\bigskip\begingroup\raggedright
  \@tempcnta=1
  \loop
    \expandafter\manuscript@printaffiliation\expandafter{\the\@tempcnta}%
    \advance\@tempcnta by1
  \ifnum\@tempcnta<\numexpr\manuscript@total+1\relax\repeat
  \par\endgroup}
\makeatother

\numberwithin{equation}{section}

\theoremstyle{definition}
\newtheorem{Definition}{Definition}[section]
\theoremstyle{plain}
\newtheorem{Theorem}[Definition]{Theorem}
\newtheorem{Proposition}[Definition]{Proposition}
\newtheorem{Lemma}[Definition]{Lemma}
\newtheorem{Corollary}[Definition]{Corollary}
\theoremstyle{definition}
\newtheorem{Remark}[Definition]{Remark}
\newtheorem{Example}[Definition]{Example}

\newcommand{\R}{\mathbb R}
\newcommand{\N}{\mathbb N}
\DeclareMathOperator{\Geo}{Geo}
\DeclareMathOperator{\TGeo}{TGeo}
\newcommand{\lm}[1]{\mathbb{L}^2(#1)}
\newcommand{\LpLSn}{Lorentzian pre-length space}
\newcommand\LpLSmaths{\ensuremath{(X, d, \ll, \leq, \tau)}}
\newcommand{\g}{\gamma}
\newcommand{\s}{\sigma}
\newcommand{\vp}{\varphi}
\newcommand{\ps}{\psi}
\newcommand{\hvp}{\hat{\varphi}}
\newcommand{\hps}{\hat{\psi}}
\let\epsilon\varepsilon

\setlist[enumerate,1]{label=(\roman*)}
\setlist[itemize]{label=$-$}

\begin{document}

\title{Timelike conjugate points in Lorentzian length spaces}
\titlemark{Timelike conjugate points in Lorentzian length spaces}

\manuscriptauthor{1}{
  \givenname{James D.E.}
  \surname{Grant}
  \zblid{grant.james-d-e}\mrid{334450}\orcid{0000-0002-0269-7193}}{J.~D.~E.~Grant}
\manuscriptauthor{2}{
  \givenname{Michael}
  \surname{Kunzinger}
  \zblid{kunzinger.michael}\mrid{348087}\orcid{0000-0002-7113-0588}}{M.~Kunzinger}
\manuscriptauthor{3}{
  \givenname{Argam}
  \surname{Ohanyan}
  \zblid{ohanyan.argam}\mrid{1498217}\orcid{0000-0002-6585-8361}}{A.~Ohanyan}
\manuscriptauthor{4}{
  \givenname{Yasmin}
  \surname{Schinnerl}
  \zblid{schinnerl.yasmin}\mrid{1473653}\orcid{0000-0002-0781-5507}}{Y.~Schinnerl}
\manuscriptauthor{5}{
  \givenname{Roland}
  \surname{Steinbauer}
  \zblid{steinbauer.roland}\mrid{616492}\orcid{0000-0001-8972-7502}}{R.~Steinbauer}

\manuscriptaffiliation{1}{
  \department{School of Mathematics and Physics}
  \organisation{University of Surrey}
  \rorid{00ks66431}\address{}\zip{GU2 7XH}\city{Guildford}
  \country{United Kingdom}
  \affemail{j.grant@surrey.ac.uk}}
\manuscriptaffiliation{2}{
  \department{Department of Mathematics}
  \organisation{University of Vienna}
  \rorid{03prydq77}
  \address{Oskar-Morgenstern-Platz 1}
  \zip{1090}\city{Wien}\country{Austria}
  \affemail{michael.kunzinger@univie.ac.at}}
\manuscriptaffiliation{3}{
  \department{Department of Mathematics}
  \organisation{University of Toronto}
  \rorid{03dbr7087}
  \address{45 St.\ George Street}
  \zip{M5S 2E5}\city{Toronto, Ontario}\country{Canada}
  \affemail{argam.ohanyan@utoronto.ca}}
\manuscriptaffiliation{4}{
  \department{Department of Mathematics}
  \organisation{University of Vienna}
  \rorid{03prydq77}
  \address{Oskar-Morgenstern-Platz 1}
  \zip{1090}\city{Wien}\country{Austria}
  \affemail{b.schinnerl@gmail.com}}
\manuscriptaffiliation{5}{
  \department{Department of Mathematics}
  \organisation{University of Vienna}
  \rorid{03prydq77}
  \address{Oskar-Morgenstern-Platz 1}
  \zip{1090}\city{Wien}\country{Austria}
  \affemail{roland.steinbauer@univie.ac.at}}

\classification[53C22, 53C23]{53B30}

\preparemanuscriptauthors
\maketitle

\begin{abstract}
We study notions of conjugate points along timelike geodesics in the synthetic setting of Lorentzian (pre-)length spaces, inspired by earlier work for metric spaces by Shankar--Sormani. After preliminary considerations on convergence of timelike and causal geodesics, we introduce and compare one-sided, symmetric, unreachable and ultimate conjugate points along timelike geodesics. We show that all such notions are compatible with the usual one in the smooth (strongly causal) spacetime setting. As applications, we prove {{a weak timelike Rauch comparison theorem}}, as well as a result closely related to the recently established Lorentzian Cartan--Hadamard theorem by Er\"{o}s--Gieger. In the appendix, we give a detailed treatment of the Fr\'{e}chet distance on the space of non-stopping curves up to reparametrization, a technical tool used throughout the paper.
\end{abstract}

\printmanuscriptmetadata

\section{Introduction}\label{section: introduction}

Lorentzian (pre-)length spaces were introduced by Kunzinger--S\"{a}mann~\cite{kunzinger2018lorentzian} as a way to study Lorentzian and spacetime geometry using only the most essential structures, akin to metric geometry in positive definite signature. Since its introduction, the synthetic theory of spacetimes has experienced rapid development, and is a very active field of research today.

Geodesics (i.e., local distance realizers) in Lorentzian pre-length spaces were introduced and studied by Grant--Kunzinger--S\"{a}mann~\cite{grant2019inextendibility}. This notion turned out to be relevant in the study of hyperbolic angles by Beran--S\"{a}mann~\cite{beran2022angles} and the recently established synthetic version of the sectional curvature variant of Bartnik's splitting conjecture by Flores--Herrera--Solis~\cite{flores2024low}.

In this article, we are interested in studying notions of conjugate points along timelike geodesics in Lorentzian (pre-)length spaces, following the fundamental contribution of Shankar--Sormani~\cite{shankar2009conjugate} on conjugate points in metric spaces. (See also the earlier work of Rinow~\cite{rinow1961innere}.) In the smooth Riemannian setting, Warner's detailed study~\cite{warner1965conjugate} of the conjugate locus sheds light on the exact behavior of the exponential map and the type of singularities that can occur. The results of Warner were extended to semi-Riemannian manifolds by Szeghy~\cite{szeghy2008conjugate}.

In Lorentzian and spacetime geometry, conjugate points are fundamental in the study of various comparison-geometric results, such as the singularity theorems of Hawking and Penrose. Our aim in this paper is to initiate the study of conjugate points in the synthetic setting for Lorentzian geometry, i.e. Lorentzian pre-length spaces. 

In the smooth setting, two points along a geodesic are~\emph{conjugate\/} if there exists a non-trivial Jacobi field along the geodesic that vanishes at both points, or equivalently by a singularity of the exponential map at both points. Jacobi fields describe infinitesimal variations through geodesics. In a Lorentzian pre-length space, neither the differential equation for Jacobi fields nor the exponential map is available. Following~\cite{shankar2009conjugate}, we instead study nearby geodesics and their dependence on their endpoints. This allows us to formulate conjugacy in terms of convergence and continuous families of geodesics, using only the causal and metric structures along with the time separation.

The different notions of conjugate points that will be considered here distinguish between two possible obstructions to solving the geodesic boundary value problem near a given timelike geodesic $\gamma$ from $p$ to $q$. \emph{One-sided conjugacy\/} detects a failure of local uniqueness with the initial point fixed, in the sense that there are points $q_i \to q$ and pairs of distinct timelike geodesics from $p$ to $q_i$, both converging to $\gamma$. \emph{Symmetric conjugacy\/} allows the initial points $p_i \to p$ to vary as well, and is therefore adapted to families of geodesics in which both endpoints are allowed to move. In particular, these notions concern pairs of geodesics approaching the same geodesic, rather than merely the existence of distinct geodesics with common endpoints.

\emph{Unreachable conjugacy\/} reflects a different obstruction, where there are perturbations of the endpoints for which no choice of connecting timelike geodesics converges to $\gamma$. Equivalently, $\gamma$ cannot be included in a timelike family that is continuous at its endpoints. \emph{Ultimate conjugacy\/} combines this failure of existence with the failure of local uniqueness expressed by symmetric conjugacy. Its absence therefore yields a timelike family continuous at $\gamma$ whose germ is unique. The distinction is relevant even in smooth strongly causal spacetimes, where one-sided, symmetric and ultimate conjugacy are all equivalent to classical conjugacy, while unreachable conjugacy implies classical conjugacy but need not follow from it.

There are additional difficulties that are specific to the Lorentzian setting. The metric $d$, which controls the topology and the Lipschitz regularity of curves, is distinct from the time separation $\tau$, whose length is locally maximized by geodesics. In the generality that we shall consider, we cannot rely on Lorentzian arclength parametrizations within the class of admissible curves. We therefore work with geodesics modulo increasing reparametrizations. The~\emph{Fr\'{e}chet distance\/} $d_F$ measures uniform proximity after such reparametrizations and makes convergence independent of the chosen representatives. We shall also need to control Lorentzian lengths, and therefore use the metric $d_\Gamma(\gamma,\sigma) \coloneqq d_F(\gamma,\sigma)+|L_\tau(\gamma)-L_\tau(\sigma)|$. Under suitable causality and local maximization assumptions, we can then establish that limits of geodesics remain geodesics, and that their lengths converge. These issues are essential both for the definitions of conjugacy and for the construction of continuous families.

{The applications show that these notions recover geometric consequences of the classical theory without the need to resort to Jacobi fields. {Our weak timelike Rauch comparison theorem (Theorem~\ref{Theorem: Timelike Rauch}) states that, in an appropriate Lorentzian pre-length space with timelike curvature bounded from above by $\kappa$, a timelike geodesic with symmetrically conjugate endpoints has length at least $D_\kappa$.}\footnote{Here $D_\kappa = \pi/\sqrt{-\kappa}$ for $\kappa < 0$ and $D_\kappa = \infty$ for $\kappa \geq 0$.} }

Theorem~\ref{Theorem: cartan hadamard adjacent} addresses the existence of nearby geodesics. For a suitable %
Lorentzian pre-length space with timelike curvature bounded above by $0$, every timelike geodesic admits a timelike family continuous at that geodesic, unique as a germ among such families. In particular, there are no ultimate conjugate points along timelike geodesics. This result follows from the synthetic Lorentzian Cartan--Hadamard theorem of Er\"{o}s--Gieger~\cite{Eroes-Gieger25}. Together, the two applications relate curvature bounds to both uniqueness and existence in the geodesic boundary value problem, providing tools for comparison geometry beyond the smooth setting.

Our article is structured as follows: Section~\ref{Section: Geodesics} recalls the necessary preliminaries and develops the convergence theory for geodesics. In particular, let $\Geo(X)$ denote the space of equivalence classes of future-directed causal geodesic segments $\gamma \colon [0, 1] \to X$ (cf.~Definition~\ref{def:GeoX}). Then, under the assumptions of Proposition~\ref{Proposition: TopologyofGeoX}, $\Geo(X)$ is closed among non-stopping curves with respect to $d_F$, and $d_F$ and $d_\Gamma$ induce the same topology on $\Geo(X)$. The required properties of the Fr\'{e}chet distance are established in Appendix~\ref{Section: appendix frechet}. Section~\ref{Section: Timelikeconjugatepoints} introduces the notions of conjugacy, studies their relation to timelike families, and discusses causal cut loci. In comparing the synthetic notions with classical conjugacy, we rely on Szeghy's~\cite{szeghy2008conjugate} generalizations of Warner's~\cite{warner1965conjugate} results. Section~\ref{Section: applications} contains the two applications described above. Finally, in Section~\ref{section: Outlook}, we give an outlook on related avenues of future research.

\section{Geodesics in Lorentzian pre-length spaces}
\label{Section: Geodesics}

\subsection{Preliminaries on Lorentzian pre-length spaces}
\label{Subsection: preliminaries on LpLS}

In this section, we collect a few basics on Lorentzian (pre-)length spaces that will be required throughout. We refer to Kunzinger--S\"{a}mann~\cite{kunzinger2018lorentzian} for full details.

\begin{Definition}[Lorentzian (pre-)length spaces]
A~\emph{Lorentzian pre-length space} is a quintuple $(X,d,\ll,\leq,\tau)$, where
\begin{enumerate}
    \item $(X,d)$ is a metric space; 
    \item $\ll, \leq$ are transitive relations (the~\emph{chronological\/} resp.\ \emph{causal relations}), $\leq$ is in addition reflexive and contains $\ll$; 
    \item $\tau \colon X^2 \to [0,\infty]$ is a lower semicontinuous function that vanishes outside of the causal relation $\leq$, the timelike relation $\ll$ coincides with $\{\tau > 0\}$ and for all $x \leq y \leq z$,
    \begin{equation}
        \tau(x,z) \geq \tau(x,y) + \tau(y,z).
    \end{equation}
\end{enumerate}
\end{Definition}

We employ standard notation $I^\pm(A), J^\pm(A)$ for chronological resp.\ causal futures/pasts of a subset $A \subseteq X$. We write $I(A,B) \coloneqq I^+(A) \cap I^-(B)$, similarly $J(A,B)$, for $A, B \subseteq X$. A subset $C \subseteq X$ is said to be~\emph{causally convex\/} if $J(x,y) \subseteq C$ whenever $x,y \in C$. A Lorentzian pre-length space is~\emph{chronological\/} if $\ll$ is irreflexive and~\emph{causal\/} if $\leq$ is a partial order.

{We write $p<q$ for $p\leq q$ with $p\neq q$. {Moreover,} $p,q$ are \emph{null-related} if $p\leq q$, $p\neq q$ and $p\not\ll q$, {i.e.,} if $p<q$ and $\tau(p,q)=0$.}

\begin{Definition}[Timelike/causal curves and $\tau$-length]
\label{def:TLcausalcurves}
A~\emph{future-directed timelike (resp.\ causal) curve\/} (past-directed curves are defined analogously) in a Lorentzian pre-length space $X$ is a nowhere constant map $\gamma \colon I=[a,b]\to X$ (i.e., $\gamma$ is not constant on any non-trivial subinterval of $I$) that
is Lipschitz (w.r.t.\ $d$) and such that for all $t_1, t_2 \in I$ with $t_1<t_2$ we have $\gamma(t_1) \ll \gamma(t_2)$ (resp.) $\gamma(t_1)\le \gamma(t_2)$ (causal curves on unbounded intervals can be defined analogously). Its~\emph{$\tau$-length\/} is defined as 
\[
L_\tau(\gamma) = \inf \left\{ \sum_{i=0}^{N-1} \tau(\gamma(t_i),\gamma(t_{i+1})) \colon 
N\in \N,\ a=t_0 < t_1 < \dots < t_N = b \right\}.
\]
A reparametrization of a future-directed causal curve $\gamma \colon [a,b] \to X$ is a future-directed causal curve $\lambda \colon [c,d] \to X$ such that $\gamma=\lambda\circ \varphi$, where $\varphi \colon [a,b]\to [c,d]$ is surjective and strictly monotonically increasing (hence also continuous). A future-directed causal curve $\gamma \colon [a,b] \to X$ is called a~\emph{maximizer\/} if $L_\tau(\gamma) = \tau(\gamma(a),\gamma(b))$.
{A causal curve is \emph{globally null} if $\gamma(s)\not\ll\gamma(t)$ whenever $s<t$. It is \emph{locally null} if every parameter has a {relative  neighborhood} on which the restriction is globally null. Global nullness is the notion of a null curve in~\cite[Def.\ 2.18]{kunzinger2018lorentzian}{, whereas} local nullness will be the relevant convention for null geodesics.}
\end{Definition}

In particular, the reparametrized curve $\lambda\circ\varphi$ above is required to itself be Lipschitz. The $\tau$-length of a causal curve is invariant under reparametrizations ({\cite[Lem.\ 2.28]{kunzinger2018lorentzian}}). We note that causal curves were defined in~\cite{kunzinger2018lorentzian} as non-constant maps rather than nowhere constant. We choose the current convention for two main reasons, one being that the restriction of a causal curve to any non-trivial subinterval of its domain should again be causal, and the other that causal curves in the standard smooth spacetime setting are indeed nowhere constant. The Fr\'{e}chet distance as detailed in Appendix~\ref{Section: appendix frechet} then applies to causal curves.

We also note that all causal curves are assumed to be future-directed unless explicitly stated otherwise.

\begin{Definition}[Model spaces]
\label{Def: model spaces}
For integers $0 \le m \le n$, denote by $\mathbb{R}_{m}^{n}$ the vector space $\mathbb{R}^{n}$ with the inner product $b(v,w) \coloneqq -\sum_{i=1}^{m}v_{i}w_{i} + \sum_{i=m+1}^{n}v_{i}w_{i},$ where $v=(v_1, \dots, v_n)$ and $w=(w_1, \dots ,w_n)$.\par
Let $K \in \mathbb{R}$. The~\emph{Lorentzian $K$-planes\/} or the~\emph{comparison spaces of constant curvature $K$\/} are the following:
\begin{itemize}
    \item For $K>0$, $\mathbb{L}^{2}(K)$ is the universal cover of $\left\{ v \in \mathbb{R}_{1}^{3} : b(v,v) = \frac{1}{K} \right\}$; 
    \item For $K<0$, $\mathbb{L}^{2}(K)$ is the universal cover of $\left\{v \in \mathbb{R}_{2}^{3} : b(v,v) = \frac{1}{K} \right\}$; 
    \item For $K=0$, $\mathbb{L}^{2}(0) \coloneqq \mathbb{R}_{1}^{2}$ is the two-dimensional Minkowski spacetime.
\end{itemize}
{The finite timelike diameter of $\lm{K}$, namely
\[
D_K \coloneqq \sup\{\bar\tau(x,y): x,y\in\lm{K},\ \bar\tau(x,y)<\infty\},
\]
is $D_K = \pi/\sqrt{-K}$ for $K < 0$ and $D_K = \infty$ for $K \ge 0$, cf.\ \cite[Def.\ 2.6]{BKR24}. } {Here we denote} the time separation of $\mathbb{L}^2(K)$ by $\bar{\tau}$.
\end{Definition}

\begin{Definition}[Alexandrov topology, strong causality]
The~\emph{Alexandrov topology\/} on a Lorentzian pre-length space $X$ is defined as the topology generated by the subbase
\begin{equation}
    \{I(x,y) : x,y \in X\}.
\end{equation}
We say $X$ is~\emph{strongly causal\/} if the Alexandrov topology agrees with the metric topology on $X$.
\end{Definition}

Following~\cite{BKR24}, we call a Lorentzian pre-length space~\emph{regular\/} if every maximizer between two timelike related points is timelike, i.e., contains no null subsegment. Moreover, a Lorentzian pre-length space is called~\emph{causally path-connected\/} if any two distinct causally related points are joined by a causal curve and any two timelike related points by a timelike curve.

\begin{Remark}[Localizability]
A~\emph{localizable\/} Lorentzian pre-length space is a Lorentzian pre-length space equipped with well-behaved local time separation functions, see~\cite[Def.\ 3.16]{kunzinger2018lorentzian} for the precise definition. It is called~\emph{regularly localizable\/} if the maximizers w.r.t.\ the local time separations (which exist by definition of localizability) are required to be  
timelike if they connect timelike related points. Regularly localizable Lorentzian pre-length spaces are regular, cf.\ \cite[Thm.\ 3.18]{kunzinger2018lorentzian}.
\end{Remark}

\begin{Proposition}[Upper semicontinuity of $L_\tau$, Prop.\ 3.17 in~\cite{kunzinger2018lorentzian}]
Let $X$ be a localizable and strongly causal Lorentzian pre-length space. Then $L_\tau$ is upper semicontinuous in the following sense: If $\gamma_n \colon [a,b] \to X$ is a sequence of causal curves converging $d$-uniformly to a causal curve $\gamma \colon [a,b] \to X$, then
\begin{equation}
    L_\tau(\gamma) \geq \limsup_{n \to \infty} L_\tau(\gamma_n).
\end{equation}
\end{Proposition}

For our definition of timelike (sectional) curvature bounds, we follow the conventions of Beran--Kunzinger--Rott~\cite[Def.\ 3.1]{BKR24}.

\begin{Definition}[Timelike curvature bounds]\label{def:timelike-curvature-bounds}
We say a Lorentzian pre-length space $X$ has~\emph{timelike curvature bounded from above (resp.\ below)\/} by $K \in \R$ if every point has a neighborhood $U$ such that
\begin{enumerate}
    \item the set $U^2 \cap \tau^{-1}([0,D_K))$ is open in $X^2$ and the restriction of $\tau$ to this set is continuous and finite-valued;
    \item whenever $x \leq y$ with $x\neq y$ and $\tau(x,y) < D_K$ and $x,y \in U$, there exists a maximizer in $U$ from $x$ to $y$;
    \item whenever $\Delta(x,y,z)$ is a timelike triangle in $U$ (i.e., $x \ll y \ll z$ with timelike maximizers running between them) and $\Delta(\bar{x},\bar{y},\bar{z})$ is the comparison triangle with equal side lengths in $\mathbb{L}^2(K)$, for any $p,q \in \Delta(x,y,z)$ and corresponding points $\bar{p},\bar{q} \in \Delta(\bar{x},\bar{y},\bar{z})$ (i.e., points with equal time separation to the endpoints of their sides as $p,q$), we have
    \begin{equation}
        \tau(p,q) \geq \bar{\tau}(\bar{p},\bar{q}) \quad (\text{resp.\ }\leq).
    \end{equation}
\end{enumerate}
{Such a neighborhood $U$ is called a comparison neighborhood. The curvature bound is called \emph{global} if all three conditions hold with $U=X$.}
\end{Definition}

\subsection{Convergence of geodesics}\label{subsection: convergenceofgeodesics}

In this section, we establish results on the convergence of geodesics in Lorentzian (pre-)length spaces. We begin by recalling the notion of geodesic introduced in~\cite[Def.\ 1.7]{beran2022angles}: 
\begin{Definition}\label{def:geodesics}
A future-directed causal curve $\gamma \colon I \to X$ is a~\emph{geodesic\/} if for each $t \in I$ there
exists a neighborhood $[a,b]$ of $t$ (i.e., $a < t < b$, allowing for equality at the
endpoints of $I$) such that $\gamma|_{[a,b]}$ is a maximizer, i.e., $\tau(\gamma(a), \gamma(b)) =
L_\tau (\gamma|_{[a,b]})$.
\end{Definition}

{Throughout this paper, a \emph{null geodesic} means a locally null causal geodesic. {It need not be globally null.}} 
Since the $\tau$-length of a causal curve $\gamma$ is parametrization-invariant, any reparametrization of a geodesic $\gamma$ is also a geodesic, and we shall henceforth consider geodesics as equivalence classes of reparametrizations. We emphasize that whenever we speak of "uniqueness" in the context of geodesics, we always mean "up to reparametrization".

\begin{Definition}[Finite $\tau$-measurability] A Lorentzian pre-length space $(X,d,\ll,\leq,\tau)$ is~\emph{finitely $\tau$-measurable\/} if every causal curve segment $\gamma \colon [a,b] \to X$ has finite $\tau$-length.
\end{Definition}

{\begin{Lemma}\label{lem:finite-tau-measurability}
Let {\LpLSmaths} be a Lorentzian pre-length space that is either localizable, or chronological {with} a timelike curvature bound (from above or below). Then it is finitely $\tau$-measurable.
\end{Lemma}
\begin{proof}
The localizable case follows from the definition of localizability, cf.\ \cite[Def.\ 3.16]{kunzinger2018lorentzian}. In the curvature-bound case, let $x\in X$ and let $U$ be a comparison neighborhood of $x$. Chronology implies $\tau(x,x)=0$. By Definition~\ref{def:timelike-curvature-bounds}(i), there is an open neighborhood $V$ of $x$ such that
\[
V\times V\subseteq U^2\cap\tau^{-1}([0,D_K)).
\]
In particular, $\tau$ is finite on $V\times V$. Given a causal curve $\eta\colon[a,b]\to X$, compactness and {Lebesgue's number lemma} provide a finite subdivision $a=t_0<\dots<t_N=b$ such that each $\eta([t_{i-1},t_i])$ lies in one such $V$. Consequently,
\[
L_\tau(\eta)\leq\sum_{i=1}^N\tau(\eta(t_{i-1}),\eta(t_i))<\infty.
\]
\end{proof}}

\begin{Remark} \label{Remark: geodesicsvssolutionsofgeodeq} 
In~\cite[Def.\ 4.1]{grant2019inextendibility}, geodesics were defined as future-directed timelike (resp.\ causal, null) curves $\gamma \colon I \to X$ on localizing Lorentzian pre-length spaces as follows: for every $t_0 \in I$ there exists a localizing neighborhood $\Omega$ of $\gamma(t_0)$ and a relatively closed interval $J \subseteq I$ with $t_0$ in its relative interior, such that $\gamma(J) \subseteq \Omega$ and $\gamma|_J$ is maximizing in $\Omega$ with respect to the local time separation function $\omega$. Thus, such curves are geodesics in the sense of Definition~\ref{def:geodesics} within each `local' Lorentzian pre-length space $\Omega$ as above. Moreover, in strongly causal localizable Lorentzian pre-length spaces both definitions agree by the proof of~\cite[Lem.\ 4.3]{grant2019inextendibility}.
\end{Remark}

\begin{Remark}[On geodesics in low causality]
The definition of causal geodesic we gave as a local $\tau$-maximizer becomes meaningless in very low causality. For example, consider the Lorentz cylinder $X = S^1 \times \R^n$ with the metric $-d\theta^2 + \sum_{i=1}^n (dx^i)^2$, then $\tau = +\infty$ on the timelike relation $\ll$. In particular, the $\tau$-length of every timelike curve segment is $+\infty$, hence every such curve is maximizing and thus (in particular) a geodesic.
\end{Remark}

An important class of examples of Lorentzian (pre-) length spaces is furnished by smooth manifolds $M$ equipped with a continuous and causally plain metric $g$ (cf.\ \cite[Prop.\ 5.8]{kunzinger2018lorentzian}). To compare the notion of geodesics in Lorentzian manifolds with that of Definition~\ref{def:geodesics}, let us now extend the validity of~\cite[Prop.\ 2.32]{kunzinger2018lorentzian}
from smooth Lorentzian metrics to a large class of continuous ones. 

\begin{Proposition}[$L_g$ vs.\ $L_\tau$]
\label{Proposition: Lg vs Ltau}
Let $(M,g)$ be a spacetime such that $g$ is continuous, causally plain and strongly causal. Then the $g$- and $\tau$-lengths of every causal curve $\gamma \colon [a, b] \to (M,g)$ agree:
\[
L_g(\gamma) = \int_a^b \sqrt{-g(\dot\gamma,\dot\gamma)} \,dt = L_\tau(\gamma).
\]
\end{Proposition}
{\begin{proof}
By strong causality and the local approximation result
\cite[Lem.\ A.1 and Prop.\ A.2]{McCann_Saemann22}, every point
has neighborhoods $V\subset W$, with $W$ causally convex for $g$,
and smooth strongly causal metrics $g_k$ on $W$ such that
$g\prec g_{k+1}\prec g_k$, $g_k\to g$ locally uniformly, and
\[
\tau_{k,W}\searrow\tau_W
\quad\text{on }V\times V.
\]
Here $\tau_W$ and $\tau_{k,W}$ are the intrinsic time separations
of $(W,g)$ and $(W,g_k)$, {respectively. Causal convexity implies}
$\tau_W=\tau|_{W\times W}$.

Let $\eta\colon[c,d]\to V$ be a causal curve segment.
For a partition $P=(c=s_0<\dots<s_m=d)$, set
\[
S_k(P):=\sum_{j=1}^m
\tau_{k,W}\bigl(\eta(s_{j-1}),\eta(s_j)\bigr).
\]
Since these sums decrease pointwise in $k$, the definition of
{length  yields}
{\[
\begin{aligned}
\lim_{k\to\infty}L_{\tau_{k,W}}(\eta)
&=\inf_k\inf_P S_k(P)
 =\inf_P\inf_k S_k(P)\\
&=\inf_P\sum_{j=1}^m
\tau_W\bigl(\eta(s_{j-1}),\eta(s_j)\bigr)
 =L_{\tau_W}(\eta).
\end{aligned}
\]}
{For each $k$, 
\cite[Prop.\ 2.32]{kunzinger2018lorentzian} shows that}
$L_{\tau_{k,W}}(\eta)=L_{g_k}(\eta)$.
{Moreover, local uniform convergence implies
$L_{g_k}(\eta)\to L_g(\eta)$.} Consequently,
\[
L_\tau(\eta)=L_{\tau_W}(\eta)
=\lim_{k\to\infty}L_{\tau_{k,W}}(\eta)
=\lim_{k\to\infty}L_{g_k}(\eta)
=L_g(\eta).
\]
Finally, subdivide $\gamma$ into finitely many segments contained
in neighborhoods $V$ of this kind and use additivity of both
length functionals.
\end{proof}
}

\begin{Remark}[Comparing notions of geodesics]
It follows from Proposition~\ref{Proposition: Lg vs Ltau} that (in strongly causal spaces) geodesics in the sense of Definition~\ref{def:geodesics} will coincide with geodesics
in the sense of solutions of the geodesic equation as long as the regularity of the metric is sufficient to allow for a satisfactory
solution theory of these ODEs and, additionally, as long as such solutions are indeed local maximizers of the $g$-length-functional.
This certainly is the case for $g\in C^{1,1}$ (cf.\ \cite{minguzzi2015convex} or~\cite{kunzinger2014regularisation}). However, 
the Hartman--Wintner-type spacetime
$M \coloneqq \R \times (-1,1) \times \R^n$ with metric $g \coloneqq -dt^2 + dx^2 + (1 - |x|^{\lambda})g_n$, where $\lambda \in (1, 2)$ and $g_n$ is the Euclidean metric on $\R^n$, is 
an example of a $C^{1,\lambda-1}$-Lorentzian metric where solutions of the geodesic equations fail to even locally maximize the $g$-length, cf.\ \cite[Ex.\ 2.4]{samann2018geodesics}. In any case, whenever we compare synthetic notions to classical notions in spacetimes we will assume strong causality in order to have a compatible notion of local maximizers.
\end{Remark}

{\begin{Lemma}[Local maximization and causal character] \ 
\label{Lemma: Geodesicshavecausalcharacter}
Let $(X,d,\ll,\leq,\tau)$ be a regular Lorentzian pre-length space. Every causal geodesic {$\gamma$} is either timelike or locally null. In the latter case its $\tau$-length is zero. In particular, a causal geodesic of positive $\tau$-length is timelike.
\end{Lemma}
\begin{proof} 
On every maximizing interval, either the endpoints are timelike related, in which case regularity makes the entire restriction timelike, or their time separation is zero. In the latter case the reverse triangle inequality makes the restriction globally null. 
The sets of parameters admitting a timelike, respectively a globally null, {maximizing  neighborhood} are relatively open, disjoint, and cover the parameter interval{, because} neighborhoods of the two different types cannot overlap in a nontrivial interval. Connectedness therefore {shows that $\gamma$ is of one type throughout.} In the timelike case, a finite subdivision and transitivity 
{imply that $\gamma(s)\ll \gamma(t)$ whenever
$s<t$.}
In the locally null case, a finite subdivision into globally null pieces and additivity give $L_\tau(\gamma)=0$.
\end{proof}

The distinction between local and global {null curves} is necessary even in smooth, globally hyperbolic spacetimes. On the flat cylinder $\R\times(\R/\mathbb Z)$ with metric $-dt^2+dx^2$, the geodesic $\gamma(s)=(s,[s])$, $0\leq s\leq1$, is locally null, but $\gamma(0)=(0,[0])\ll(1,[0])=\gamma(1)$.}

Before discussing the limit curve theorem and its implications for geodesics, we adapt some material from~\cite{hau2020causal} to our setting.

\begin{Definition}[Local relations]
Let $(X,d,\ll,\leq,\tau)$ be a Lorentzian pre-length space and $U \subseteq X$ open. We define the local timelike and causal relations on $U$ as follows: For $p,q \in U$ we say $p \ll_U q$ ($p \leq_U q$) if there exists a future-directed timelike (causal) curve $\gamma \colon [a, b] \to X$ with $\gamma(a) = p$, $\gamma(b) = q$ and $\gamma \left( [a, b] \right) \subseteq U$. For the causal relation we also include $p\leq_U p$ for every $p\in U$.
\end{Definition}

\begin{Lemma}
Let {\LpLSmaths} be a causally path-connected {\LpLSn}. Then, for any causally convex $V \subseteq X$, we have $\le_V \, = \left. \le \right|_{V \times V}$.
\end{Lemma}

\begin{proof}
If $p, q \in V$ with $p \le_V q$, then either $p=q$ or there exists a future-directed, causal curve from $p$ to $q$, hence $p \le q$, so $(p, q) \in \left. \le \right|_{V \times V}$.

Conversely, let $p, q \in V$ with $p \le q$. The case $p=q$ is immediate, so assume $p\neq q$. Since {\LpLSmaths} is causally path-connected, there exists a causal curve $\gamma \colon [a, b] \to X$ with $\gamma(a) = p$, $\gamma(b) = q$. Since $V$ is causally convex, $\gamma \left( [a, b] \right) \subseteq V$. Therefore, $p \le_V q$.
\end{proof}

\begin{Definition}[Local weak causal closedness]
A set $U \subseteq X$ is~\emph{weakly causally closed\/} if, for all sequences $\left( p_n \right)_n$ and $\left( q_n \right)_n$ in $U$ such that $p_n \le_U q_n$, for all $n \in \N$ with $p_n \to p \in U$ and $q_n \to q \in U$, one has $p \le_U q$. A Lorentzian pre-length space is~\emph{locally weakly causally closed\/} if every point has a weakly causally closed neighborhood.
\end{Definition}

Note that, unlike~\cite[Definitions~2.18 \& 2.19]{hau2020causal}, we do not require that {\LpLSmaths} be causally path-connected in our definition of $\le_U$ and, hence, of weakly causally closed. Our formulation, where the local causal relations are determined by the curves under consideration, is more appropriate for the limit curve theorem below (Theorem~\ref{Theorem: limitcurvetheoremforsegments}) and the resulting convergence theory for geodesics (Proposition~\ref{Proposition: TopologyofGeoX}). The following slightly sharpened version of~\cite[Prop.~2.21]{hau2020causal} isolates the relation between $\le_U$ and $\leq$ that is needed to compare the two closedness conditions.

\begin{Lemma}
\label{Lemma:2.21}
Let $\LpLSmaths$ be a strongly causal Lorentzian pre-length space with the property that, for any causally convex $V \subseteq X$, we have $\le_V \, = \left. \le \right|_{V \times V}$. Then, the following are equivalent: 
\begin{enumerate}
\item\label{221i} $\LpLSmaths$ is locally causally closed;%
\footnote{Recall from~\cite[Definition~3.4]{kunzinger2018lorentzian} that if {\LpLSmaths} is a {\LpLSn}, then a subset $U \subseteq X$ is~\emph{causally closed\/} if, for all sequences $(p_n)$, $(q_n)$ in $U$ such that $p_n \le q_n$ for all $n$  with $p_n \to p \in \bar{U}$, $q_n \to q \in \bar{U}$, we have $p \le q$. A {\LpLSn} is~\emph{locally causally closed\/} if every point has a causally closed neighborhood.}
\item\label{221ii} $\LpLSmaths$ is locally weakly causally closed.
\end{enumerate}
\end{Lemma}

\begin{proof}
{\ }\\[-5mm]
\begin{itemize}[fullwidth]
(i) $\Rightarrow$ (ii): Let $x \in X$ and $U$ a causally closed neighborhood of $x$. By strong causality, there exists a causally convex neighborhood $V$ of $x$ with $x \in V \subseteq U$. Let $(p_n)_n$ and $(q_n)_n$ be sequences in $V$ with $p_n \le_V q_n$ for all $n$, such that $p_n \to p \in V$ and $q_n \to q \in V$ as $n \to \infty$. Since $V$ is causally convex, $\le_V \, = \left. \le \right|_{V \times V}$, so $p_n \le q_n$. Therefore, since $U$ is causally closed and $p, q \in V \subseteq U$, we have $p \le q$. Since $p, q \in V$ and $\le_V \, = \left. \le \right|_{V \times V}$, it follows that $p \le_V q$. Hence, $V$ is a weakly causally closed neighborhood of $x$. 

\item[(ii) $\Rightarrow$ (i)]: Let $x \in X$ and $U$ a weakly causally closed neighborhood of $x$. Since $(X, d)$ is a metric space, it satisfies the $T_3$ separation condition. In combination with strong causality, this implies that there exists a causally convex neighborhood $V$ of $x$ with the property that $x \in V \subseteq \overline{V} \subseteq U$. Let $(p_n)_n$ and $(q_n)_n$ be sequences in $V$ with $p_n \le q_n$ for all $n$, such that $p_n \to p$ and $q_n \to q$, where $p, q \in \overline{V} \subseteq U$. Since $V$ is causally convex, we have $p_n \le_V q_n$ and therefore, since $V \subseteq U$, $p_n \le_U q_n$ for all $n \in \N$. Since $U$ is weakly causally closed, it follows that $p \le_U q$ and, therefore, $p \le q$. Therefore, $V$ is a causally closed neighborhood of $x$.
\qedhere
\end{itemize}
\end{proof}

Finally, as pointed out in~\cite[pp.~8]{hau2020causal}, the limit curve theorem~\cite[Thm.\ 3.7]{kunzinger2018lorentzian} holds with the assumption of local weak causal closedness.%

\begin{Theorem}[Limit curve theorem for segments]
\label{Theorem: limitcurvetheoremforsegments}
Let $(X,d,\ll,\leq,\tau)$ be a 
locally weakly causally closed Lorentzian pre-length space. Let $\gamma_n \colon [a, b] \to X$ be a uniformly Lipschitz sequence of future-directed causal curves. Suppose that the $\gamma_n$ are either
\begin{itemize}
    \item[(a)] all contained in a common compact set, or 
    \item[(b)] the metric $d$ is proper, and there is $t_0 \in [a, b]$ such that a subsequence of $\gamma_n(t_0)$ converges. 
\end{itemize}
Then a subsequence of the $\gamma_n$ converges $d$-uniformly to a Lipschitz curve $\gamma \colon [a, b] \to X$ with $\gamma(s) \leq \gamma(t)$ for $s < t$. 
Moreover:
\begin{enumerate}
\item\label{LCi} If $\gamma$ is non-constant (in particular, if its endpoints are distinct), then its normalized $d$-arclength parametrization $\hat\gamma\colon[0,1]\to X$ is a future-directed causal curve with the same endpoints. A reparametrization of the convergent subsequence converges uniformly to $\hat\gamma$.
\item\label{LCii} If $X$ is, in addition, strongly causal and localizable, $\gamma$ is non-constant, and the $\gamma_n$ are limit maximizing, i.e.\ there exists a sequence $\varepsilon_n \to 0$ such that
\[
L_{\tau}(\gamma_n) \geq \tau(\gamma_n(a),\gamma_n(b)) - \varepsilon_n,
\]
then ${\hat\gamma}$ is maximizing. In this case, $L_{\tau}({\hat\gamma}) = \lim_{k \to \infty} L_{\tau}(\gamma_{n_k})$, where $\gamma_{n_k}$ is the subsequence converging to $\gamma$.
\end{enumerate}
\end{Theorem}
\begin{proof}
From the proof of~\cite[Thm.~3.7]{kunzinger2018lorentzian}, we deduce that there exists a subsequence $\left( \gamma_{n_k} \right)_k$ of curves that converge uniformly to a Lipschitz continuous limit curve $\gamma$. Therefore, without loss of generality, we may assume that the sequence of curves $\left( \gamma_n \right)_n$ converges uniformly to $\gamma$. Since $X$ is locally weakly causally closed, for every $t \in [a, b]$ there exists an open, weakly causally closed neighborhood $U_t$ of $\gamma(t)$. Choose a compact interval $I_t\subseteq[a,b]$ containing $t$ in its relative interior such that $\gamma(I_t)\subseteq U_t$. Uniform convergence implies that $\gamma_n(I_t)\subseteq U_t$ for all sufficiently large $n$. Thus, for $t_1<t_2$ in $I_t$, local weak causal closedness gives $\gamma(t_1)\leq_{U_t}\gamma(t_2)$, hence $\gamma(t_1)\leq\gamma(t_2)$. Subdividing any subinterval of $[a,b]$ into finitely many such intervals and using transitivity proves that $\gamma$ preserves the causal order.

For~\ref{LCi}, put $\ell=L_d(\gamma)>0$ and $s(t)=L_d(\gamma|_{[a,t]})/\ell$. The usual arclength construction gives $\gamma=\hat\gamma\circ s$, where $s$ is Lipschitz, non-decreasing and surjective, and $\hat\gamma\colon[0,1]\to X$ is Lipschitz and nowhere constant. Since $\gamma$ preserves the causal order, so does $\hat\gamma$, which is therefore causal. To obtain convergence to this representative, set
\[
h_n(t)=\frac{s(t)+(t-a)/(n(b-a))}{1+1/n}.
\]
Each $h_n\colon[a,b]\to[0,1]$ is an increasing bi-Lipschitz bijection, and $h_n\to s$ uniformly. Consequently, the causal curves $\gamma_n\circ h_n^{-1}$ converge uniformly to $\hat\gamma$. This construction removes constant intervals of the limit without admitting them in the definition of causal curves.

The claim~\ref{LCii} about the maximality of the limit is a consequence of lower semicontinuity of $\tau$ and upper semicontinuity of $L_{\tau}$ applied to $\gamma_n\circ h_n^{-1}\to\hat\gamma$, cf.\ the proof of~\cite[Thm.\ 2.23]{beran2023splitting}.
\end{proof}

In the context of geodesics, we will often be interested in Lorentzian pre-length spaces where the local maximization property is uniform. This motivates the following definitions (one of which was already introduced in~\cite[Def.\ 2.20]{beran2023splitting}).

\begin{Definition}[Locally maximizing]\label{def:locally-maximising}
A Lorentzian pre-length space $(X,d,\ll,\leq,\tau)$ is called~\emph{locally timelike (causally, null) maximizing}\footnote{Such spaces were referred to as~\emph{locally quasiuniformly maximizing} in~\cite[Def.\ 2.20]{beran2023splitting}.} if for each point $p \in X$ there exists a neighborhood $U$ of $p$ such that each timelike (causal, null) geodesic segment entirely contained in $U$ is maximizing.
\end{Definition}

If $X$ is locally causally maximizing, we simply say that it is locally maximizing.

\begin{Example}[$C^{1,1}$-spacetimes are locally maximizing]
Let $(M,g)$ be a strongly causal $C^{1,1}$-spacetime, then it is locally maximizing: By Remark~\ref{Remark: geodesicsvssolutionsofgeodeq}, the two available notions of timelike (null, causal) geodesics agree. Moreover, by \cite[Thm.\ 6]{minguzzi2015convex} local maximizers coincide with classical solutions of the geodesic equation.
Now let $p \in M$ and $U$ a convex neighborhood containing $p$ (cf.\ \cite{kunzinger2014regularisation,minguzzi2015convex}). Then 
for any two distinct causally related points $q_1 \leq q_2$, $q_1,q_2\in U$, there is a unique causal geodesic connecting them which is entirely contained in $U$. Hence any causal geodesic in $U$ is already maximizing.
\end{Example}

\begin{Example}[Low regularity spacetimes which are not locally maximizing] The Hartman-Wintner-type spacetimes from \cite[Ex.\ 2.4]{samann2018geodesics}, which are $C^{1,\alpha}$ for $0<\alpha<1$, are not locally maximizing.
\end{Example}

\begin{Example}
    The Lorentzian constant curvature $K$ spaces, which are (the universal coverings of) de-Sitter space, Minkowski space, and anti-de Sitter space of the corresponding radius (cf.\ Definition~\ref{Def: model spaces}), are locally maximizing: Indeed in the case of $K \geq 0$ this is trivial since causal geodesics are maximizing throughout, and in the case $K < 0$ one can consider convex neighborhoods of a given point $x$ which avoid the timelike cut locus 
    $\{\tau(x,.) = \frac{\pi}{\sqrt{-K}}\}$.
\end{Example}

For the study of conjugate points, we require a notion of convergence that controls not only the proximity of geodesics up to reparametrization, but, in addition, their Lorentzian lengths. The Lorentzian length does not, in general, converge under uniform convergence of curves alone, due to the length functional being only upper semicontinuous. Motivated by Shankar--Sormani~\cite[Def.\ 1.5]{shankar2009conjugate}, we define a metric $d_{\Gamma}$ that encodes both the Fr\'{e}chet distance between curves and the difference between their lengths. (See~\eqref{Frechet1} below.)

\begin{Definition}[Spaces of geodesics]
\label{def:GeoX}
Let $(X,d,\ll,\leq,\tau)$ be a finitely $\tau$-measurable Lorentzian pre-length space. We denote by $\Geo(X)$ the set of equivalence classes of future-directed causal geodesic segments $\gamma \colon [0, 1] \to X$, where two such segments are called equivalent if and only if there exists a  strictly increasing and surjective reparametrization $\varphi \colon [0,1] \to [0,1]$ between them.%
\footnote{See Definition~\ref{def:TLcausalcurves}.} %
We denote by $\TGeo(X) \subseteq \Geo(X)$ the subset of equivalence classes of future-directed timelike geodesic segments. For convenience, we write $\gamma \in \Geo(X)$ for both the equivalence class and some representative. 

We consider the following metrics on $\Geo(X)$: $d_F$ is the restriction of the Fr\'{e}chet distance (see Appendix~\ref{Section: appendix frechet}), and
\begin{equation}
    d_{\Gamma}(\gamma_1,\gamma_2) \coloneqq d_F(\gamma_1,\gamma_2) + |L_{\tau}(\gamma_1) - L_{\tau}(\gamma_2)|.
\label{Frechet1}
\end{equation}
\end{Definition}

Equivalently, $d_\Gamma$ is the metric induced on the graph of $L_\tau$ by the taxicab product metric on $(\Geo(X),d_F)\times\R$. Related product metrics are used by Barrera--Montes de Oca--Solis in their study of compact balls in metric length spaces~\cite{BMS21balls} and of compact diamonds in Lorentzian length spaces~\cite{BMS24diamonds}. In our setting, the following proposition gives conditions under which the additional length term does not change the topology on $\Geo(X)$.

\begin{Proposition}[Topology of $\Geo(X)$]
\label{Proposition: TopologyofGeoX}
Let $(X,d,\ll,\leq,\tau)$ be a localizable, locally weakly causally closed, strongly causal, and locally maximizing Lorentzian pre-length space. Then $\Geo(X)$ is a closed subset of the set of nowhere constant curves from $[0,1]$ to $X$ w.r.t.\ the Fr\'{e}chet distance $d_F$. Moreover, the metric topologies of $d_F$ and $d_\Gamma$ coincide on $\Geo(X)$.
\begin{proof}
    Let $(\gamma_n) \subseteq \Geo(X)$ converge, with respect to $d_F$, to the equivalence class of a nowhere constant path $\gamma \colon [0,1] \to X$. Fix this representative $\gamma$ of the limit and choose a causal representative $\eta_n$ of each $\gamma_n$. By the definition of $d_F$ in Appendix~\ref{Section: appendix frechet},
\[
 d_F([\eta_n],[\gamma])=\inf_h\sup_{t\in[0,1]}d(\eta_n(h(t)),\gamma(t)),
\]
where $h$ ranges over the increasing homeomorphisms of $[0,1]$. Hence we may choose such homeomorphisms $h_n$ with
\[
 \sup_{t\in[0,1]}d(\eta_n(h_n(t)),\gamma(t))<d_F([\eta_n],[\gamma])+\frac{1}{n}.
\]
Although this already gives uniform convergence of $\eta_n\circ h_n$ to $\gamma$, these compositions need not be Lipschitz, since the $h_n$ need not be Lipschitz. To obtain causal representatives, approximate each $h_n$ uniformly by a strictly increasing piecewise linear homeomorphism $p_n$, using linear interpolation on a sufficiently fine finite partition of $[0,1]$. By uniform continuity of $\eta_n$, the approximation can be chosen so that
\[
 \sup_{t\in[0,1]}d(\eta_n(p_n(t)),\eta_n(h_n(t))) < \frac{1}{n}.
\]
The triangle inequality then gives
\[
 \sup_{t\in[0,1]}d(\eta_n(p_n(t)),\gamma(t))<d_F([\eta_n],[\gamma]) + \frac{2}{n} \longrightarrow 0.
\]
Each $\eta_n\circ p_n$ is Lipschitz as a composition of Lipschitz maps, preserves the causal order since $p_n$ is increasing, and is nowhere constant since $p_n$ maps every non-trivial interval onto a non-trivial interval. Thus these compositions are causal representatives converging uniformly to $\gamma$; we again denote them by $\gamma_n$. Their Lipschitz constants need not be bounded uniformly in $n$. The local closedness argument in the proof of Theorem~\ref{Theorem: limitcurvetheoremforsegments} applies to the uniformly convergent sequence of causal representatives and shows that $\gamma$ preserves the causal order; no Lipschitz bound on the sequence is needed for this step. The set $K=\gamma([0,1])\cup\bigcup_n\gamma_n([0,1])$ is compact: for a sequence of points on curves with indices tending to infinity, compactness of the parameter interval and uniform convergence give a subsequence converging to a point of $\gamma([0,1])$, while bounded indices reduce to a finite union of compact images. By strong causality, cover $K$ by finitely many causally convex neighborhoods $V_i$, each contained in a localizing neighborhood $\Omega_i$. Localizability supplies a bound $C_i$ on the $d$-lengths of causal curves contained in $\Omega_i$. A causal curve cannot leave and then re-enter $V_i$, so its portion in $V_i$ has $d$-length at most $C_i$. Summing these bounds gives $L_d(\gamma_n)\leq\sum_i C_i$ for every $n$. Lower semicontinuity of metric length therefore yields $L_d(\gamma)\leq\liminf_n L_d(\gamma_n)<\infty$. As in the proof of Theorem~\ref{Theorem: limitcurvetheoremforsegments}\,\ref{LCi}, the normalized $d$-arclength parametrization $\hat\gamma$ of $\gamma$ is a future-directed causal curve. Since $\gamma$ is nowhere constant, its accumulated length is strictly increasing, so $\hat\gamma$ represents the same equivalence class as $\gamma$. By the definition of $d_F$, we may again choose causal representatives of $\gamma_n$ converging uniformly to $\hat\gamma$, using piecewise linear increasing reparametrizations as at the beginning of this proof. We use these representatives below and again denote them by $\gamma_n$ and $\gamma$. Given $t \in [0,1]$, let $U$ be a neighborhood containing $\gamma(t)$ such that each causal geodesic entirely contained in $U$ is maximizing. W.l.o.g.\ suppose $t \in (0,1)$ and suppose $\delta >0$ so that $\gamma([t-\delta,t+\delta]) \subseteq U$. By uniform convergence, $\gamma_n([t-\delta,t+\delta])\subseteq U$ for all large enough $n$. Since $L_{\tau}$ is upper semicontinuous, one may then argue just like in the case of the limit curve theorem (cf.\ Theorem~\ref{Theorem: limitcurvetheoremforsegments}) to conclude that $\gamma([t-\delta,t+\delta])$ is maximizing, and in fact $\lim_n L_{\tau}(\gamma_n|_{[t-\delta,t+\delta]}) = L_{\tau}(\gamma|_{[t-\delta,t+\delta]})$ (cf.\ again the proof of~\cite[Thm.\ 2.23]{beran2023splitting}). Since $t$ was arbitrary, we conclude that $\gamma \in \Geo(X)$. Due to the convergence of Lorentzian lengths we just pointed out, we also see that $d_{F}$ and $d_{\Gamma}$ induce the same topology on $\Geo(X)$.
\end{proof}
\end{Proposition}

\medskip
We shall also require the following useful variant of the limit curve theorem. (See~\cite[Thm.\ 2]{EVP21} for a related result.) To this end, recall that a locally weakly causally closed, causally path-connected and localizable Lorentzian pre-length space is called a~\emph{Lorentzian length space\/} if
\begin{equation}
    \tau(x,y) = \sup( \{L_\tau(\gamma) : \gamma \colon [a,b] \to X \text{ causal with }\gamma(a) = x, \, \gamma(b) = y\} \cup \{0\}).
\end{equation}
A Lorentzian length space is~\emph{globally hyperbolic\/} if it is non-totally imprisoning (i.e., $d$-lengths of causal curves in compact sets are uniformly bounded) and $J(x,y)$ is compact for all $x,y \in X$. Globally hyperbolic Lorentzian length spaces are strongly causal, cf.\ \cite[Thm.\ 3.26]{kunzinger2018lorentzian}.

\begin{Theorem}[Limit curve theorem: Converging endpoints]
\label{Theorem: Limitcurvetheoremtwoendpointcase}
Let $(X,d,\ll,\leq,\tau)$ be a globally hyperbolic Lorentzian length space. Let $\gamma_n \colon [0,1] \to X$ be a sequence of future-directed causal curves such that $p_n \coloneqq \gamma_n(0) \to p$ and $q_n \coloneqq \gamma_n(1) \to q$ with $p\ne q$. Then a subsequence of a reparametrized sequence $(\tilde{\gamma}_n)$ converges uniformly to a future-directed causal curve $\gamma$ from $p$ to $q$. After passing to this subsequence and relabelling, we have:
\begin{enumerate}
\item\label{217i} If each $\gamma_n$ is a geodesic and $X$ is locally maximizing, then $\gamma$ is also a geodesic, and in this case $L_\tau(\gamma) = \lim_n L_\tau(\gamma_n)$. Moreover, if $X$ is regular and $\lim_n L_\tau(\gamma_n)>0$, then $\gamma$ is a future-directed timelike geodesic.
\item\label{217ii} If $(\gamma_n)$ is a limit maximizing sequence (in particular: if each $\gamma_n$ is maximizing), then $\gamma$ is maximizing.
\end{enumerate}
\begin{proof}
    Choose $\tilde p\in I^-(p)$ and $\tilde q\in I^+(q)$. For all sufficiently large $n$, the curves $\gamma_n$ are contained in the compact diamond $J(\tilde p,\tilde q)$. Non-total imprisonment bounds their $d$-lengths uniformly. Their normalized $d$-arclength parametrizations on $[0,1]$ therefore form a uniformly Lipschitz sequence. Theorem~\ref{Theorem: limitcurvetheoremforsegments} gives a uniformly convergent subsequence whose limit has the distinct endpoints $p,q$. By part~\ref{LCi} of that theorem, removing any constant intervals of the limit and reparametrizing the subsequence gives uniform convergence to a future-directed causal curve $\gamma\colon[0,1]\to X$.
   
     Part~\ref{217i} now follows from Proposition~\ref{Proposition: TopologyofGeoX} and the assumption that $p \neq q$. {The additional claim follows from Lemma~\ref{Lemma: Geodesicshavecausalcharacter} {since} positive $\tau$-length excludes the locally null case.}

    Part~\ref{217ii} follows from Theorem~\ref{Theorem: limitcurvetheoremforsegments}, which states that the curve $\gamma$ is maximizing if the sequence $\gamma_n$ is limit maximizing.
\end{proof}
\end{Theorem}

\begin{Remark}[Limit curve theorem for $\Geo(X)$]
\label{Remark: limitcurveconvendp for Geo}
We may reformulate Theorem~\ref{Theorem: Limitcurvetheoremtwoendpointcase} in the language of equivalence classes as follows: Let $X$ be a globally hyperbolic and locally maximizing Lorentzian length space and $\gamma_n \in \Geo(X)$. Suppose that $\gamma_n(0) \eqqcolon p_n \to p \neq q \leftarrow q_n \coloneqq \gamma_n(1)$. Then a subsequence of $(\gamma_n)$ converges in $\Geo(X)$ to some  $\gamma$. Moreover, if $X$ is regular and $L_\tau(\gamma)>0$, then $\gamma \in \TGeo(X)$. If each $\gamma_n$ is maximizing, then so is $\gamma$.
\end{Remark}

We end this subsection with a standard property of geodesic convergence in smooth semi-Riemannian manifolds, recalled for use in Proposition~\ref{prop:embeddable_for_spacetimes}.
\begin{Lemma}\label{lem:geod_convergence}
Let $(M,g)$ be a smooth semi-Riemannian manifold. Let $a < 0 \leq 1 < b$. Then the following are equivalent for affinely parametrized geodesic segments $\gamma_n, \gamma \colon [a,b] \to M$:
\begin{enumerate}
    \item $\gamma_n \to \gamma$ in $C^0$.
    \item $\gamma_n \to \gamma$ in $C^\infty$.
    \item $\gamma_n'(0) \to \gamma'(0)$ in $TM$.
\end{enumerate}
\end{Lemma}
\begin{proof}
Clearly, $(ii)$ implies $(i)$. To see that $(iii)$ implies $(ii)$, observe that if $(p_n,v_n) \in TM$ with $(p_n,v_n) \to (p,v)$, then the fact that $\gamma_n(t) \coloneqq \exp_{p_n}(t v_n)$ converges to $\gamma(t) \coloneqq \exp_{p}(tv)$ in $C^\infty_{\mathrm{loc}}$ on the maximal domain of $\gamma$ is a consequence of basic ODE theory, see e.g.\ \cite[Cor.\ 2.6]{KOSS}.
To prove $(i)\Rightarrow(iii)$, choose a convex normal neighborhood $U$ of $\gamma(0)$ and $h\in(0,1)$ sufficiently small that $\gamma([0,h])\subseteq U$. Uniform convergence then implies that $\gamma_n([0,h])\subseteq U$ for all sufficiently large $n$. Let $\exp_U^{-1}(x,y)\in T_xM$ denote the initial velocity of the unique geodesic in $U$ from $x$ to $y$, parametrized on $[0,1]$. This is the smooth local inverse of $\exp(v)=(\pi(v),\exp_{\pi(v)}(v))$, where $\pi\colon TM\to M$ is the bundle projection. Since $s\mapsto\gamma_n(hs)$ is a geodesic in $U$ with initial velocity $h\gamma_n'(0)$, uniqueness gives
\[
\exp_U^{-1}(\gamma_n(0),\gamma_n(h))=h\gamma_n'(0).
\]
The same identity holds for $\gamma$. Continuity of $\exp_U^{-1}$ and of scalar multiplication in $TM$ therefore yields
\[
\gamma_n'(0)
=\frac{1}{h}\exp_U^{-1}(\gamma_n(0),\gamma_n(h)) 
\longrightarrow\frac{1}{h}\exp_U^{-1}(\gamma(0),\gamma(h))=\gamma'(0),
\]
concluding the proof.
\end{proof}

\section{Timelike conjugate points}
\label{Section: Timelikeconjugatepoints}

This section is dedicated to the discussion of several notions of conjugate points in the synthetic setting. These notions were introduced in positive definite signature by Shankar--Sormani~\cite{shankar2009conjugate} (cf.\ also Rinow~\cite{rinow1961innere}). We will for the most part follow their terminology and conventions when adapting the notions to the Lorentzian setting.

\subsection{Symmetric conjugate points}
\label{Subsection: symmetric conjugate points}

For the following discussion, let $(X,d,\ll,\leq, \tau)$ be a finitely $\tau$-measurable Lorentzian pre-length space.

\begin{Definition}[One-sided conjugate points]\label{def:one-sided-conjugacy}
Let $\gamma \in \TGeo(X)$ be a future-directed timelike geodesic from $p$ to $q$. We say $q$ is~\emph{one-sided conjugate to} $p$ if there is a sequence $q_i \to q$ and for each $i$, there exist two distinct future-directed timelike geodesics $\eta_i \neq \gamma_i \in \TGeo(X)$ from $p$ to $q_i$ such that $ \eta_i,\gamma_i \to \gamma$ with respect to $d_{\Gamma}$.
\end{Definition}

\begin{Theorem}\label{th:one-sided-vs-classical}
In a strongly causal smooth spacetime $(M,g)$, points $p\ll q$ are one-sided conjugate points along a future-directed timelike geodesic $\gamma$ if and only if they are conjugate points (in the Jacobi field sense).
\end{Theorem}

{\begin{proof}
We parametrize $\gamma$ affinely on $[0,1]$ and set $v=\dot\gamma(0)$, so that $\gamma(t)=\exp_p(tv)$ and $\exp_p(v)=q$.
Suppose first that $q$ is one-sided conjugate to $p$ along $\gamma$ and choose sequences of timelike geodesics, parametrized affinely on $[0,1]$ as in Definition \ref{def:one-sided-conjugacy}. Then by
Lemma \ref{lem:geod_convergence} their initial vectors converge to $v$. If $p,q$ were not conjugate along $\gamma$, the inverse function theorem would make $\exp_p$ injective on a neighborhood of $v$. {However, for all sufficiently large indices the two initial vectors belong to that neighborhood and have the same image, hence coincide, a contradiction.}

For the reverse implication we use Szeghy's local extension~\cite{szeghy2008conjugate} of Warner's theorem~\cite{warner1965conjugate}. Suppose that $q=\gamma(1)$ is conjugate to $p$ along $\gamma$, and let $C(p)$ denote the set of conjugate vectors in the domain of $\exp_p$. {In particular,} $v\in C(p)$. Define
\[
\mathcal D_v=\{J'(1): J\text{ is a Jacobi field along }\gamma,\ J(0)=J(1)=0\},
\]
where the prime denotes covariant differentiation. For every such $J$, the function $g(J,\dot\gamma)$ is affine and vanishes at both endpoints, hence is identically zero. Therefore
{$\mathcal D_v\subseteq\dot\gamma(1)^\perp$.}
Since $\dot\gamma(1)$ is timelike, $g|_{\mathcal D_v}$ is positive definite. By~\cite[Lem.\ 4.1]{szeghy2008conjugate}, there is a radially convex open neighborhood $U$ of $v$, contained in the domain of $\exp_p$ and consisting of future-directed timelike vectors, on which $\exp_p$ is a Warner map.

We may now apply~\cite[Thm.\ 2.1]{szeghy2008conjugate} {to conclude that the} regular conjugate vectors $C^R(p)\cap U$ form an open dense subset of $C(p)\cap U$. Every open neighborhood $W\subseteq U$ of $v$ therefore contains a regular conjugate vector $v'$. The local normal forms in~\cite[Thm.\ 2.2]{szeghy2008conjugate} {then} imply that $\exp_p$ is not injective on any neighborhood of $v'$, and hence is not injective on $W$.

Applying this to a decreasing neighborhood basis of $v$ contained in $U$, choose $v_i\neq w_i$ with $v_i,w_i\to v$ and
\[
\exp_p(v_i)=\exp_p(w_i)=:q_i\longrightarrow q.
\]
The curves $\gamma_i(t)=\exp_p(tv_i)$ and $\sigma_i(t)=\exp_p(tw_i)$ are distinct future-directed timelike geodesics from $p$ to $q_i$. Again by Lemma \ref{lem:geod_convergence}, both $\gamma_i$ and $\sigma_i$ converge to $\gamma$ in $C^\infty([0,1])$, in particular uniformly and in Lorentzian length. Since $L_g=L_\tau$ in this smooth, strongly causal setting, both sequences converge to $\gamma$ in $d_\Gamma$, implying one-sided conjugacy.
\end{proof}}

Next, we discuss a notion of conjugate point similar to one-sided conjugacy, but here the initial point is also allowed to vary.

\begin{Definition}[Symmetric conjugate points]
Let $p\ll q$. Then $p,q$ are called~\emph{symmetrically conjugate} along a future-directed timelike geodesic $\gamma \in \TGeo(X)$ running between them if there are $p_i \to p$ and $q_i \to q$ such that for each $i$ there are two distinct future-directed timelike geodesics $\gamma_i \neq \sigma_i \in \TGeo(X)$ from $p_i$ to $q_i$ such that $\gamma_i \to \gamma$ and $\sigma_i \to \gamma$ with respect to $d_{\Gamma}$.
\end{Definition}
Clearly, any one-sided conjugate point is also symmetrically conjugate. The converse is open even in the metric case, cf.\ \cite[Open Problem 2.5]{shankar2009conjugate}.

\begin{Remark}
The existence of two distinct timelike geodesics with the same endpoints does not by itself imply one-sided or symmetric conjugacy. Both sequences in the definitions must converge to the same geodesic; constant sequences given by two distinct geodesics do not have this property. For example, on the flat Lorentzian cylinder $M=\R\times(\R/\mathbb{Z})$ with metric $-dt^2+dx^2$, the curves
\[
\gamma(s)=(Ts,[s/2]),\qquad \sigma(s)=(Ts,[-s/2]),\qquad s\in[0,1],
\]
with $T>1/2$ are distinct future-directed timelike geodesics from $(0,[0])$ to $(T,[1/2])$. However, there are no conjugate points along these geodesics. Since the cylinder is strongly causal, Proposition~\ref{prop:equivalence_of_conjugate_notions} also excludes one-sided and symmetric conjugacy.
\end{Remark}

\begin{Proposition}
\label{prop:equivalence_of_conjugate_notions}
Let $(M,g)$ be a strongly causal smooth spacetime and let $\gamma$ be a future-directed timelike geodesic from $p$ to $q$. Then the following are equivalent:
\begin{enumerate}
    \item[(i)] $q$ is one-sided conjugate to $p$ along $\gamma$.
    \item[(ii)] $p,q$ are symmetrically conjugate along $\gamma$.
    \item[(iii)] $q$ is a conjugate point of $p$ along $\gamma$ in the classical sense.
\end{enumerate}
\end{Proposition}
\begin{proof}
By Theorem~\ref{th:one-sided-vs-classical}, (i)$\Leftrightarrow$(iii). As observed above, one-sided conjugate implies symmetric conjugate, i.e.\ (i)$\Rightarrow$(ii). Finally, (ii)$\Leftrightarrow$(iii) follows by the same 
proof as in~\cite[Thm.\ 2.6]{shankar2009conjugate}.
\end{proof}

\subsection{Families of timelike geodesics}
\label{subsection: Families of timelike geodesics}

\begin{Definition}[Continuous timelike families]
Let $\gamma\in \TGeo(X)$ be a future-directed timelike geodesic from $p=\gamma(0)$ to $q=\gamma(1)$. A map $F \colon U\times V\to \TGeo(X)$
with $U$ a neighborhood of $p$ and $V$ a neighborhood of $q$ is called a~\emph{timelike family about $\gamma$\/} if 
$U\ll V$ (i.e., $p' \ll q'$ for all $p' \in U$ and $q' \in V$) and $F(p',q')$ goes from $p'$ to $q'$ for each $(p',q')\in U\times V$
and $F(p,q)=\gamma$. 
We say that $F$ is~\emph{continuous about\/} $\gamma$ if $F$ is a continuous (w.r.t.\ $d_{\Gamma}$) map.
$F$ is called~\emph{continuous at\/} $\gamma$ if it is only required to be continuous (w.r.t.\ $d_{\Gamma}$) at the point $(p,q)$ itself.
\end{Definition}
Observe that since $\ll$ is an open relation (cf.~\cite[Prop.\ 2.13]{kunzinger2018lorentzian}), whenever $x\ll y$ there exist neighborhoods $U$ of $x$ and $V$ of $y$ such that $p' \ll q'$ for all $p' \in U$ and $q' \in V$.

\begin{Remark}
To relate this terminology to selection theory, consider the
set-valued map
\[
    G(x,y)
    :=
    \{\eta\in\operatorname{TGeo}(X):
      \eta(0)=x,\ \eta(1)=y\},
    \qquad x\ll y.
\]
A timelike family about $\gamma$ satisfies
$F(x,y)\in G(x,y)$ for all $(x,y)\in U\times V$, together
with the prescribed value $F(p,q)=\gamma$.
If the family is continuous about $\gamma$, then $F$ is
a selection of $G|_{U\times V}$ in Michael's terminology,
which includes continuity in the definition of a selection;
see \cite[\S~1, pp.~361--362]{Michael1956}.
Equivalently, such a family is a continuous local section
of the endpoint map
\[
    \operatorname{ev}:\operatorname{TGeo}(X)
    \longrightarrow \{(x,y)\in X\times X:x\ll y\},
    \qquad
    \operatorname{ev}(\eta)=(\eta(0),\eta(1)),
\]
passing through $\gamma$.
We retain the term \emph{timelike family} to distinguish
continuity on the entire domain from continuity only
at $(p,q)$, as the latter alone does not make $F$ a selection in the sense of Michael.
\end{Remark}

\begin{Definition}[Unique timelike families]\label{def:unique_familiy}
Let $\gamma$ be a future-directed timelike geodesic. We say that there exists a unique timelike
family which is continuous about $\gamma$ resp.\ at $\gamma$ if
\begin{itemize}
\item[(i)] there exists a timelike family $F \colon U\times V \to \TGeo(X)$ that is continuous about $\gamma$ resp.\ at $\gamma$, and
\item[(ii)] if $F' \colon U' \times V' \to \TGeo(X)$ is another timelike family about $\gamma$ that is continuous about $\gamma$ resp.\ at $\gamma$, then there exist neighborhoods $U''\subseteq U\cap U'$ and $V''\subseteq V\cap V'$ of $\gamma(0)$
resp.\ $\gamma(1)$ such that $F$ coincides with $F'$ on $U''\times V''$.
\end{itemize}
\end{Definition}

Another way to express this is to say that there exists a unique germ of timelike
families around $(\gamma(0), \gamma(1))$ with the respective continuity property. Given a timelike family $F \colon U\times V \to \TGeo(X)$ about $\gamma$, we shall denote by $[F]$ the germ around $(\gamma(0), \gamma(1))$ determined by $F$. Note that if a timelike family is continuous about $\gamma$ and unique among timelike families that are continuous at $\gamma$, then it is also unique among timelike families that are continuous about $\gamma$.

A $C^0$-continuous family is a family as above that is continuous w.r.t.\ $d_F$ instead of $d_{\Gamma}$. As we saw before, if $X$ is localizable, locally weakly causally closed, strongly causal and locally maximizing, then these two notions coincide, cf.\ Proposition~\ref{Proposition: TopologyofGeoX}. 

The following basic observation will repeatedly be used below.

\begin{Lemma}\label{lem:no-unique-fam-conj-points} Let $\gamma\in \TGeo(X)$ connect the points $p$ and $q$ and suppose that there exists a timelike family $F \colon U\times V \to \TGeo(X)$ continuous at $\gamma$. If there isn't a unique timelike family continuous at $\gamma$, then $p$ and $q$ are symmetrically conjugate. 
\end{Lemma}
\begin{proof} Our assumptions mean that (the second alternative in) (i) in Definition~\ref{def:unique_familiy} is satisfied, but (ii) is not. Hence there exists some timelike family $F' \colon U'\times V' \to \TGeo(X)$ continuous at $\gamma$ that does not coincide with $F$ on sufficiently small neighborhoods of $(p,q)$. For each $i\in \N$ pick neighborhoods $U_i\subseteq U\cap U'$ of $p$ and $V_i\subseteq V\cap V'$ such that $F|_{U_i\times V_i} \ne F'|_{U_i\times V_i}$, as well as
$U_i\subseteq B^d_{1/i}(p)$ and $V_i\subseteq B^d_{1/i}(q)$. Furthermore, by continuity of $F$ and $F'$ at $(p,q)$ we may assume 
$U_i$ and $V_i$ small enough to grant that $F(U_i\times V_i)
\subseteq B^{d_\Gamma}_{1/i}(\gamma)$ and also $F'(U_i\times V_i)
\subseteq B^{d_\Gamma}_{1/i}(\gamma)$. Pick $(p_i,q_i) \in U_i\times V_i$ such that
$\gamma_i \coloneqq F(p_i,q_i)\ne F'(p_i,q_i) \eqqcolon \sigma_i$. Then $\gamma_i$ and $\sigma_i$ are timelike
geodesics as required to establish that $p$ and $q$ are symmetrically conjugate.
\end{proof}

As an immediate consequence we obtain the following analogue of~\cite[Thm.\ 3.7]{shankar2009conjugate}: 

\begin{Corollary}\label{cor:not-sym-con-unique-germ}
Let $\gamma\in \TGeo(X)$ connect $p$ and $q$ and suppose there exists a timelike family $F \colon U\times V \to \TGeo(X)$ continuous at $\gamma$. Then (i) implies (ii):
\begin{itemize}
\item[(i)] $p$ and $q$ are not symmetrically conjugate along $\gamma$.
\item[(ii)] $F$ is a representative of the unique germ of timelike families continuous at $\gamma$.
\end{itemize}
\end{Corollary}

\begin{Definition}[Embeddability] A sequence $(\gamma_i)_{i\in \N}$ in $\TGeo(X)$ connecting points $p_i$ and $q_i$ and converging (with respect to $d_\Gamma$) to a curve $\gamma\in \TGeo(X)$ is~\emph{embeddable\/} if there exists a timelike family $F \colon U\times V \to \TGeo(X)$ which is continuous at $\gamma$ such that, for $i$ large, $(p_i,q_i)\in U\times V$ and $F(p_i,q_i) = \gamma_i$. It is~\emph{strongly embeddable\/} if the timelike family can be chosen to be continuous about $\gamma$.
\end{Definition}
In particular a constant sequence $(\gamma)$ is embeddable resp.\ strongly embeddable if and only if there exists a timelike family about $\gamma$ that is continuous at $\gamma$ resp.\ about $\gamma$.
In this case we say that $\gamma$ itself is (strongly) embeddable. Embeddability therefore asks for a selection with the prescribed values $F(p_i,q_i)=\gamma_i$ for all sufficiently large $i$, continuous at $(p,q)=(\gamma(0),\gamma(1))$. Strong embeddability requires continuity on a neighborhood of $(p,q)$.

\begin{Proposition}\label{prop:embed-char}
Let $\gamma\in \TGeo(X)$ connect $p$ to $q$. Then the following are equivalent:
\begin{itemize}
\item[(i)] $\gamma$ is embeddable.
\item[(ii)] For all sequences $p_i\to p$ and $q_i\to q$ there exists $i_0\in \N$ such that for all $i\ge i_0$ there is a geodesic $\gamma_i\in \TGeo(X)$ from $p_i$ to $q_i$ such that 
$\gamma_i\to \gamma$ with respect to $d_\Gamma$.
\end{itemize}
\end{Proposition}
\begin{proof}
(i)$\Rightarrow$(ii): Let $F \colon U \times V \to \TGeo(X)$ be a timelike family continuous at $\gamma$. Then given sequences as in (ii), for $i$ large
we may set $\gamma_i \coloneqq F(p_i,q_i)$. \\
(ii)$\Rightarrow$(i): We first show that for each $k\in \N$ there exist neighborhoods $U_k$ of $p$ and $V_k$ of $q$ such that for all $(p',q')\in U_k\times V_k$ there exists a timelike geodesic $\sigma\in \TGeo(X)$ from $p'$ to $q'$ with $d_\Gamma(\gamma,\sigma)<\frac{1}{k}$. Indeed, if this were not the case, then for some $k\in \N$ there would exist a sequence $(p_i,q_i)\to (p,q)$ such that
any timelike geodesic connecting $p_i$ to $q_i$ would have $d_\Gamma$-distance from $\gamma$ at least $\frac{1}{k}$, contradicting (ii).
Since the relation $\ll$ is open, we can arrange that $U_k\ll V_k$. Furthermore, we may assume that $U_{k+1}\subseteq U_k$ and $V_{k+1}\subseteq V_k$ for all $k$ and that $\bigcap_k U_k = \{p\}$, as well as $\bigcap_k V_k = \{q\}$. Now set $W_k \coloneqq (U_k\times V_k)\setminus (U_{k+1}\times V_{k+1})$ for each $k$ and define 
\[
F \colon U_1 \times V_1 \to \TGeo(X)
\]
as follows: Set $F(p,q) \coloneqq \gamma$, and for each $(p',q')\in (U_1\times V_1)\setminus \{(p,q)\}$ take the unique
$k\in \N$ such that $(p',q')\in W_k$ and define $F(p',q')$ to be any timelike geodesic from $p'$ to $q'$ satisfying $d_\Gamma(\gamma,F(p',q'))<\frac{1}{k}$. Then $F$ is a timelike family, and it is also continuous at $\gamma$: if $(p_i,q_i)\to (p,q)$, then for any $k\in \N$ there exists some $i_0$
such that $(p_i,q_i)\in U_k\times V_k$ for all $i\ge i_0$, implying that $d_\Gamma(\gamma,F(p_i,q_i))<\frac{1}{k}$.
\end{proof}

There is also a sheaf theoretic interpretation of continuous timelike families: the continuous local sections of $\mathrm{ev}$ form a sheaf on the open set $\{(x,y):x\ll y\}$. Sections that agree on overlaps glue uniquely, since continuity and the endpoint condition are local. (See~\cite[Section~6.7]{StacksSheaves}.) Note that the next Lemma does not follow directly from the sheaf gluing axiom. 

\begin{Lemma} Suppose that every $d_\Gamma$-convergent sequence $(\gamma_i)_{i\in \N}$ in $\TGeo(X)$ is embeddable.
If for every future-directed timelike geodesic $\gamma \in \TGeo(X)$ such that $L_{\tau}(\gamma) < R$ there exists a unique timelike family about $\gamma$ that is continuous at $\gamma$, then it is even true that for every future-directed timelike geodesic $\gamma$ with $L_{\tau}(\gamma) < R$ there exists a timelike family that is continuous about $\gamma$.
\end{Lemma}
\begin{proof}
Let $F \colon U \times V \to\TGeo(X)$ be a timelike family about a future-directed timelike geodesic $\gamma \colon [0,1] \to X$ that is continuous at $(\gamma(0),\gamma(1))$. By $d_{\Gamma}$-continuity of $F$ at $(\gamma(0),\gamma(1))$, we can shrink $U$ and $V$ to assume w.l.o.g.\ that all geodesics $F(x,y)$ satisfy $L_{\tau}(F(x,y)) < R$, $x \in U$, $y \in V$.\\
Suppose there are no neighborhoods $U',V'$ of $\gamma(0),\gamma(1)$ (with $U'\ll V'$) such that $F|_{U' \times V'}$ is continuous. Then there are sequences $p_i \to \gamma(0)$ and $q_i \to \gamma(1)$ such that $F$ is not continuous at any $(p_i,q_i)$. Let $\gamma_i \coloneqq F(p_i,q_i)$. So for any $i$ there are $\varepsilon_i>0$ and sequences $p_{ij} \to p_i$ and $q_{ij} \to q_i$ such that
\begin{equation}
\label{eq:nolim}
    d_\Gamma(F(p_{ij}, q_{ij}),F(p_i,q_i))>\varepsilon_i.
\end{equation}
Since $L_{\tau}(\gamma_i) < R$, for each $\gamma_i$ there is a timelike family $F_i$ about $\gamma_i$ continuous at its endpoints $(p_i,q_i)$. In particular,
\[
\lim_j F_i(p_{ij},q_{ij}) = F_i(p_i,q_i)=\gamma_i = F(p_i,q_i).
\]
So by~\eqref{eq:nolim} for each $i$ there is $j_i$ such that for all $j \geq j_i$,
\[
d_\Gamma(F(p_{ij},q_{ij}), F_i(p_{ij},q_{ij})) > \varepsilon_i/2.
\]
We can choose $j_i$ larger to guarantee that 
\[
d_{\Gamma}(F_i(p_{ij_i},q_{ij_i}),\gamma_i) < 1/i
\]
due to the convergence noted above, and also $(p_{ij_i},q_{ij_i}) \to (\gamma(0),\gamma(1))$ by a diagonal argument. Thus we have two distinct sequences of future-directed timelike geodesics, namely $F(p_{ij_i},q_{ij_i})$ and $F_i(p_{ij_i},q_{ij_i})$ both converging to $\gamma$, where for the second convergence we use
\[
d_\Gamma(F_i(p_{ij_i}, q_{i,j_i}),\gamma)\leq 1/i+d_\Gamma(F(p_i,q_i),F(p,q)).
\]
Now embed the sequence $(F_i(p_{ij_i}, q_{i,j_i}))_i$ of geodesics into a timelike family $F'$ that is continuous at $\gamma$. Then $[F]\ne [F']$, contradicting our assumption. 
\end{proof}

\begin{Remark} The previous result is a direct analogue of~\cite[Lem.\ 3.4]{shankar2009conjugate}, although our argument requires the additional assumption of embeddability. While this assumption does not explicitly appear in~\cite{shankar2009conjugate}, it is not clear to us whether or not it can be removed, in general.
\end{Remark}

\begin{Proposition}\label{prop:embeddable_for_spacetimes} Let $(M,g)$ be a strongly causal smooth spacetime. Let $(\gamma_i)_{i\in\mathbb{N}}$ be a sequence of timelike geodesics in $M$ converging (with respect to $d_{\Gamma}$) to {a timelike limit geodesic $\gamma\in\TGeo(M)$}. Let $p = \gamma(0)$ and $q = \gamma(1)$. 
If $p$ and $q$ are not conjugate along $\gamma$, then the sequence $(\gamma_i)$ is strongly embeddable.
\end{Proposition}
\begin{proof} 
Choose affine parametrizations on $[0,1]$ and let $v_i=\gamma_i'(0)$ and $v=\gamma'(0)$. We first justify that $d_\Gamma$-convergence implies $v_i\to v$, without using non-conjugacy. Choose representatives $\eta_i$ converging uniformly to $\gamma$, and a convex normal neighborhood $U$ containing $\gamma([0,h])$ for some small $h>0$. For large $i$, we also have $\eta_i([0,h]) \subseteq U$. Put $p_i=\eta_i(0)$, $z_i=\eta_i(h)$, $L_i=L_\tau(\gamma_i)$ and $\ell_i=L_\tau(\eta_i|_{[0,h]})$. Local maximality and continuity of the time separation in $U$ give $\ell_i\to hL_\tau(\gamma)>0$. Since affine parameter is proportional to Lorentzian arclength, the local inverse used in the proof of Lemma~\ref{lem:geod_convergence} gives
\[
v_i=\frac{L_i}{\ell_i}\exp_U^{-1}(p_i,z_i)
\longrightarrow \frac{1}{h}\exp_U^{-1}(p,\gamma(h))=v.
\]
We write $\exp\colon\mathcal{D}\subseteq TM\to M\times M$ for the exponential map of the geodesic flow, and $\exp_p\colon\mathcal{D}\cap T_pM\to M$ for the exponential map based at $p$. We then have 
\[ 
\exp(w) = (\pi(w), \exp_{\pi(w)}(w)), 
\]
where $\pi\colon TM\to M$ is the bundle projection and $\mathcal{D}$ is the open set of initial velocities whose geodesics are defined on $[0,1]$. Then the differential of $\exp$ at $v$ takes the following lower-triangular block form:
\[
T_v\exp \cong \begin{pmatrix}
\mathrm{id} & 0 \\
\star & T_v\exp_p
\end{pmatrix}.
\]
As $p$ and $q$ are not conjugate along $\gamma$, $T_v\exp_p$ is invertible, implying that $\exp$ is a diffeomorphism from a neighbourhood $W \subset TM$ of $v$ onto a neighbourhood $U \times V$ of $(p, q)$ in $M\times M$. Let $H \colon U \times V \to W$ be the smooth inverse of $\exp|_{W}$.

We define the timelike family $F \colon U \times V \to \TGeo(M)$ by
\[ 
F(x, y)(t) \coloneqq \exp_x(t \cdot H(x, y)). 
\]
This family is continuous about $\gamma$. Moreover, since $\gamma_i \to \gamma$, for sufficiently large $i$ the initial velocities $v_i$ lie in $W$ and the endpoints $(p_i, q_i)$ lie in $U \times V$. Consequently, $H(p_i, q_i) = v_i$, and $F(p_i, q_i)$ is exactly the geodesic $\gamma_i$. Thus, the sequence is strongly embeddable.
\end{proof}

The following result is a direct analogue of \cite[Prop.\ 3.5]{shankar2009conjugate}.
\begin{Theorem}\label{thm:classical-conj-vs-unique-family}
Let $(M,g)$ be a smooth strongly causal spacetime and let $\gamma \colon [0, 1] \to M$ be a future-directed timelike geodesic connecting $p$ to $q$. Then the following are equivalent:
\begin{itemize}
\item[(i)] $p$ and $q$ are not conjugate (in the classical sense) along $\gamma$.
\item[(ii)] There exists a timelike family which is continuous about $\gamma$ and which is
unique among all timelike families that are continuous at $\gamma$.
\end{itemize}
\end{Theorem} 
\begin{proof}
(i)$\Rightarrow$(ii): Applying Proposition~\ref{prop:embeddable_for_spacetimes} to the constant sequence $(\gamma)$
it follows immediately from the proof of that result that the timelike
family constructed there is continuous about $\gamma$ and is unique among all such families.

(ii)$\Rightarrow$(i): Let $F\colon U\times V\to\TGeo(M)$ be continuous about $\gamma$, with $U,V$ open. Parametrize its geodesics affinely on $[0,1]$ and put $H(y)=F(p,y)'(0)\in T_pM$. The initial-velocity convergence argument in the preceding proof shows that $H$ is continuous. Since $\exp_p(H(y))=y$, the map $H$ is injective. By invariance of domain, $H(V)$ is an open neighborhood of $\gamma'(0)$, and $\exp_p$ is injective there. Theorem~\ref{th:one-sided-vs-classical} and its proof therefore imply that $p$ and $q$ are not conjugate along $\gamma$.
\end{proof}

\begin{Proposition}[Existence of families around unique maximizers]
\label{Proposition: Existenceoffamiliesifgeodunique}
Let $(X,d,\ll,\leq,\tau)$ be a regular, globally hyperbolic and locally maximizing Lorentzian length space. Let $p \ll q$ and suppose there is a unique maximizer $\gamma \in \TGeo(X)$ from $p$ to $q$. Then there exists a timelike family about $\gamma$ consisting of maximizers which is continuous at $\gamma$, in particular, $\gamma$ is embeddable.
\begin{proof}
    Let $U, V$ be neighborhoods of $p,q$ so that $x \ll y$ for all $x \in U$, $y \in V$. For each $(x,y) \in U \times V$, let $F(x,y)$ be a maximizer from $x$ to $y$, hence $F(p,q)=\gamma$. We claim that $F$ is continuous at $\gamma$. Indeed, if $(p_i,q_i)\to (p,q)$, then by Theorem~\ref{Theorem: Limitcurvetheoremtwoendpointcase} (see also Remark~\ref{Remark: limitcurveconvendp for Geo}) and regularity of $X$, 
    there exists a subsequence of $(F(p_i,q_i))$ that converges (in $d_{\Gamma}$) to a timelike maximizer connecting $p$ to $q$. By uniqueness, this maximizer has to coincide with $\gamma$. Since the same is true for any subsequence of $(F(p_i,q_i))$, we must have that $d_\Gamma(F(p_i,q_i),F(p,q))\to 0$ as $i\to \infty$.  
\end{proof}
\end{Proposition}

\subsection{Unreachable and ultimate conjugate points}
\label{Subsection: unreachable and ultimate}

We continue to assume that $X$ is a finitely $\tau$-measurable Lorentzian pre-length space.

\begin{Definition}[Unreachable conjugate points]
Let $\gamma \in \TGeo(X)$ with $\gamma(0) = p$ and $\gamma(1) = q$. We say that $p,q$ are~\emph{unreachable conjugate along\/} $\gamma$ if $\gamma$ is not embeddable.
Otherwise, $p$ and $q$ are called~\emph{reachable\/} along $\gamma$.
\end{Definition}

Proposition~\ref{prop:embed-char} immediately yields the following characterization:
\begin{Proposition}\label{prop:unreachable-characterization} 
Let $\gamma\in \TGeo(X)$ connect $p$ and $q$. Then the following are equivalent:
\begin{itemize}
\item[(i)] $p$ and $q$ are unreachable conjugate along $\gamma$; 
\item[(ii)] {There exist sequences $p_i\to p$ and $q_i\to q$ for which there is no index $i_0\in\N$ and no choice of timelike geodesics $\gamma_i$ from $p_i$ to $q_i$ for every $i\geq i_0$ such that $\gamma_i\to\gamma$ with respect to $d_\Gamma$.}
\end{itemize}
\end{Proposition}
\begin{Lemma}\label{lem:unreachable_implies_standard_in_spacetimes}
Let $(M,g)$ be a strongly causal
spacetime. If $p\ll q$ are unreachable conjugate along a timelike geodesic $\gamma$, then they are conjugate (in the classical sense) along $\gamma$. 
\end{Lemma}
\begin{proof}
If $p$ and $q$ were not conjugate in the classical sense, then
by Theorem~\ref{thm:classical-conj-vs-unique-family} there would exist a
timelike family about $\gamma$ that is continuous at $\gamma$, contradicting
Proposition~\ref{prop:unreachable-characterization} (ii).
\end{proof}

{\begin{Example}
Consider the universal cover $\widetilde H^{1,n}(K)$ of the anti-de Sitter quadric $H^{1,n}(K)$, with $K<0$ and conjugate radius $D_K=\pi/\sqrt{-K}$. Write $\pi\colon\widetilde H^{1,n}(K)\to H^{1,n}(K)$ for the covering projection. If $p\ll q$ and $\tau(p,q)=D_K$, then $p,q$ are unreachable conjugate along any timelike maximizer $\gamma$ between them.

Indeed, the projected geodesic $\pi\circ\gamma$ lies in the intersection of the quadric with a two-dimensional linear plane $P$, and $\pi(q)=-\pi(p)$. Choose another such plane $Q\neq P$ whose intersection with the quadric is timelike. Its future-directed geodesic lift starting at $p$ reaches $q$ after proper time $D_K${, and we denote} this lift by $\sigma$. Choose interior points $p_i,q_i$ on $\sigma$ tending respectively to $p,q$, with $p\ll p_i\ll q_i\ll q$. The projected points $\pi(p_i),\pi(q_i)$ are linearly independent and {uniquely} determine $Q$. Hence the segment of $\sigma$ between $p_i$ and $q_i$ is the unique future-directed timelike geodesic between these lifted endpoints. These segments converge in $d_\Gamma$ to $\sigma$, which is distinct from $\gamma$, so no choice of connecting geodesics can converge to $\gamma$. {The claim now follows from Proposition~\ref{prop:unreachable-characterization}.}
\end{Example}}

\begin{Corollary}\label{cor:unique-max-reachable}
    Let $(X,d,\ll,\leq,\tau)$ be a regular, globally hyperbolic and locally maximizing Lorentzian length space. Let $p \ll q$ and suppose there exists a unique maximizer $\gamma \in \TGeo(X)$ from $p$ to $q$. Then $p,q$ are reachable.
\begin{proof}
    This follows from Proposition~\ref{Proposition: Existenceoffamiliesifgeodunique}.
\end{proof}
\end{Corollary}

\begin{Definition}[Ultimate conjugate points]
Let $(X,d,\ll,\leq,\tau)$ be a finitely $\tau$-measurable Lorentzian pre-length space and $\gamma \in \TGeo(X)$. Then $\gamma(0),\gamma(1)$ are called~\emph{ultimate conjugate points\/} if they are unreachable conjugate points or symmetric conjugate points (along $\gamma$).
\end{Definition}

\begin{Example} In~\cite[Ex.\ 4.7]{shankar2009conjugate}, an example (due to V.\ Schroeder) of a geodesic $\sigma \colon [0,1] \to N$ in a
complete Riemannian manifold $(N,h)$ (specifically, a two-dimensional ellipsoid) is given, with the following properties: $\sigma$ is the {unique minimizing geodesic} in $N$ connecting $x=\sigma(0)$ and $y=\sigma(1)$, while $x$ and $y$ are conjugate along $\sigma$. We may ``Lorentzify'' this example as follows: Let $M \coloneqq \R\times N$ with Lorentzian metric $g=-dt^2 + h_x$. Then $(M,g)$ is globally hyperbolic (each slice $\{t=t_0\}$ is a Cauchy surface), and the following correspondences hold (cf., e.g., \cite[Ch.\ 7]{ON83}): Any geodesic in $(M,g)$ is of the form $s\mapsto (as+b,\beta(s))$, where $a,b\in \R$ and $\beta$ is a geodesic in $(N,h)$. Jacobi fields in $(M,g)$ take the form $s\mapsto (T(s),J(s))$, where $T$ is affine linear and $J$ is a Jacobi field in $(N,h)$. Thus if we set $\gamma(s) \coloneqq (as,\sigma(s))$ with $\sigma$ as above, then for $a>0$ sufficiently large, $\gamma$ is a timelike geodesic connecting $p=(0,x)$ to $q=(a,y)$. {These points are conjugate along $\gamma$. Since $\sigma$ is the unique minimizing geodesic from $x$ to $y$ in $N$, its affine lift $\gamma$ is the unique timelike maximizer from $p$ to $q$ in $(M,g)$ (note that
$\tau((0,x),(a,y))=\sqrt{a^2-d_h(x,y)^2}$).} By Corollary~\ref{cor:unique-max-reachable}, $p$ and $q$ are reachable along $\gamma$. Consequently, they provide an example of points along a timelike geodesic that are (by Proposition~\ref{prop:equivalence_of_conjugate_notions}) symmetrically (hence ultimate) conjugate, but not unreachable conjugate.
\end{Example}

\begin{Corollary}
Let $(M,g)$ be a smooth, strongly causal spacetime and $\gamma \colon [0,1] \to M$ a timelike geodesic. Then $\gamma(0),\gamma(1)$ are conjugate in the sense of Jacobi fields if and only if they are ultimate conjugate.
\begin{proof} This is immediate from Proposition~\ref{prop:equivalence_of_conjugate_notions} and Lemma 
\ref{lem:unreachable_implies_standard_in_spacetimes}.
\end{proof}
\end{Corollary}

\begin{Proposition}
    Let $(X,d,\ll,\leq,\tau)$ be a regular, locally maximizing, globally hyperbolic Lorentzian length space. Let $\gamma \in \TGeo(X)$ such that $p \coloneqq \gamma(0)$ and $q \coloneqq \gamma(1)$ are not ultimate conjugate along $\gamma$. Then there exists a timelike family around $\gamma$ which is continuous at $\gamma$ and whose germ is unique among all timelike families continuous at $\gamma$.
    \begin{proof}
        By assumption, $p,q$ are reachable and not symmetric conjugate. By definition of reachability, there exists a timelike family $F \colon U \times V \to \TGeo(X)$ continuous at $\gamma$. By Corollary~\ref{cor:not-sym-con-unique-germ}, $F$ is a representative of the unique germ of timelike families continuous at $\gamma$.
    \end{proof}
\end{Proposition}

\subsection{Causal cut loci}
\label{Subsection: causalcutloci}

{\begin{Definition}[Causal cut points]
Let $(X,d,\ll,\leq,\tau)$ be a finitely $\tau$-measurable Lorentzian pre-length space and let $p,q\in X$ be distinct. We say $q$ is a~\emph{future causal cut point\/} of $p$ if there exist at least two distinct maximizing elements in $\Geo(X)$ connecting $p$ to $q$. If $X$ is regular, we call such a cut point~\emph{timelike\/} if $p\ll q$ and~\emph{null\/} if $p\leq q$ and $p\not\ll q$. We denote the corresponding sets by $K_C(p)$, $K_T(p)$ and $K_N(p)$.
\end{Definition}

In a regular space, the maximizers in the timelike case are timelike, while those in the null case are globally null, since $\tau(p,q)=0$. Thus $K_C(p)=K_T(p)\mathbin{\dot\cup}K_N(p)$.}

{\begin{Definition}[Initial cut loci]
\label{Definition: firstcutlocus}
Let $(X,d,\ll,\leq,\tau)$ be a finitely $\tau$-measurable Lorentzian pre-length space and $p \in X$. Given $\gamma \in \Geo(X)$, we say $q \coloneqq \gamma(1)$ is an~\emph{initial future causal cut point\/} of $p=\gamma(0)$ along $\gamma$ if
\begin{enumerate}
    \item there exist future causal cut points $q_i$ of $p$ converging to $q$;
    \item there is no point $q' \in \gamma([0,1]) \setminus \{q\}$ with property (i).
\end{enumerate}
More explicitly, along any causal geodesic $\gamma$ starting at $p$, take its first intersection with $\overline{K_C(p)}$: if
\[
A_\gamma=\{t\in[0,1]:\gamma(t)\in\overline{K_C(p)}\}
\]
is nonempty, the initial future causal cut point is $\gamma(\min A_\gamma)$. This includes $p$ itself when $0\in A_\gamma$. If $A_\gamma=\varnothing$, no initial future causal cut point is assigned. The set of such points is the~\emph{initial future causal cut locus} $\mathrm{IniCut}_C^+(p)$.

If $X$ is regular, the~\emph{initial future timelike cut locus} $\mathrm{IniCut}_T^+(p)$ and the~\emph{initial future null cut locus} $\mathrm{IniCut}_N^+(p)$ are defined in the same way, using timelike geodesics and $\overline{K_T(p)}$, respectively locally null geodesics and $\overline{K_N(p)}$, with the same conventions.
\end{Definition}

In general, $\mathrm{IniCut}_C^+(p)$ need not be the disjoint union of the timelike and null initial loci. For example, in the flat cylinder $M=\R\times(\R/\mathbb Z)$ with metric $-dt^2+dx^2$, put
\[
p=(0,[0]),\qquad q=(1/2,[1/2]),\qquad
X=I_M^+(p)\cup\eta([0,1/2]),\quad \eta(s)=(s,[s]).
\]
With the restricted metric, causal relations and time separation, $X$ is a regular, finitely $\tau$-measurable Lorentzian pre-length space. Lengths are inherited from $M$, and every $X$-maximizer is an $M$-maximizer. Here
\[
K_C(p)=K_T(p)=\{(t,[1/2]):t>1/2\},
\]
since the two cylinder maximizers $u\mapsto(tu,[\pm u/2])$, $0\leq u\leq1$, have interiors in $I_M^+(p)$. Thus $q\in\mathrm{IniCut}_C^+(p)$. However, $K_N(p)=\varnothing$, since each point null-related to $p$ lies on $\eta$ which is the unique causal curve there. As $p\not\ll q$, we obtain $q\notin\mathrm{IniCut}_T^+(p)\cup\mathrm{IniCut}_N^+(p)$.}

{\begin{Remark}[Cut points in smooth spacetimes]
Let $(M,g)$ be a smooth, globally hyperbolic spacetime and let $\gamma\colon[0,1]\to M$ be a future causal geodesic from $p$ to $q$ that is maximizing up to $q$. Then $q$ is its future cut point in the standard sense if and only if there exists another maximizer from $p$ to $q$ or $q$ is the first conjugate point to $p$ along $\gamma$ (cf., e.g., \cite[Prop.\ 3.2.28]{Treude:2011}).

The three initial cut loci in Definition~\ref{Definition: firstcutlocus} agree with the usual future causal, timelike and null cut loci in this setting. In the causal case this follows from density of multiple-maximizer points and closedness of the classical cut locus (cf., e.g., \cite[Prop.\ 3.2.30 and Thm.\ 3.2.31(iii)]{Treude:2011}). The timelike case follows by restriction to the open set $I^+(p)$. For the null case, regard $H=\partial I^+(p)\setminus\{p\}$ as a horizon in $M\setminus\{p\}$ in the sense of 
\cite[Def.\ 1]{szeghy2023horizons}. By \cite[Prop.\ 3]{szeghy2023horizons}, its generator (future) endpoints lie in the closure of its points with multiple generators. These are precisely the classical null cut points and the points of $K_N(p)$, respectively. Closedness of the classical null cut locus and taking first intersections along null geodesics give the result. Consequently, $\mathrm{IniCut}_C^+(p)=\mathrm{IniCut}_T^+(p)\mathbin{\dot\cup}\mathrm{IniCut}_N^+(p)$.
\end{Remark}}

\begin{Remark}[Comparison with~\cite{shankar2009conjugate}]
In~\cite[Def.\ 7.2]{shankar2009conjugate}, the authors require the $q_i$ from Definition~\ref{Definition: firstcutlocus} (in their metric setting) to lie on $\gamma$. 
This notion, however, does not coincide with the standard notion on smooth Riemannian manifolds.
Indeed, the two-dimensional torus of revolution provides a counterexample: The cut locus of a point $p$ slightly above the outer equator contains an arc lying on a parallel whose endpoints are conjugate but cannot be reached by more than one minimizer from $p$, see \cite{GMST:05}, (iii) on p.\ 104 and Fig.\ 4. Consider one of these (two) points, call it $q$ and suppose it is reached from $p$ by a geodesic $\gamma$. Since $\gamma$ beyond $q$ cannot coincide with the arc it does not contain any meeting points of minimizers from $p$ (although $q$ is a limit of such points lying on the arc).
\end{Remark}

\begin{Definition}[Initial injectivity radii]
Let $(X,d,\ll,\leq,\tau)$ be a finitely $\tau$-measurable Lorentzian pre-length space. For $p \in X$, we define the~\emph{initial future causal injectivity radius} of $X$ at $p$ to be $\mathrm{IniInj}_C^+(p) \coloneqq d(p,\mathrm{IniCut}_C^+(p))$. The~\emph{initial future causal injectivity radius} of $X$ is then $\mathrm{IniInj}_C^+(X) \coloneqq \inf_{p \in X} \mathrm{IniInj}_C^+(p)$. If $X$ is in addition regular, we define in complete analogy the~\emph{initial future timelike/null injectivity radii} $\mathrm{IniInj}_T^+(p),\mathrm{IniInj}_T^+(X), \mathrm{IniInj}_N^+(p)$ and $\mathrm{IniInj}_N^+(X)$.
Here and throughout, we use the convention $d(p,\varnothing)=\inf\varnothing=\infty$. Thus $\mathrm{IniInj}_C^+(p)=\infty$ if $\mathrm{IniCut}_C^+(p)$ is empty, and the same convention applies to the timelike and null injectivity radii.
\end{Definition}

\begin{Definition}[Unique injectivity radii]
\label{def:unique-injectivity-radii}
Let $(X,d,\ll,\leq,\tau)$ be a finitely $\tau$-measurable and geodesic Lorentzian pre-length space. Then the~\emph{future causal unique injectivity radius} $\mathrm{UniqueInj}_C^+(X)$ is the supremum over all $r\geq0$ such that all {distinct} causally related points $p,q\in X$ with $d(p,q)<r$ are joined by unique maximizers. If $X$ is in addition regular, we define the~\emph{future timelike/null unique injectivity radii} $\mathrm{UniqueInj}_T^+(X)$ and $\mathrm{UniqueInj}_N^+(X)$ analogously, restricting to timelike-related and null-related pairs, respectively.
\end{Definition}

The initial and unique injectivity radii satisfy the following one-sided comparison.

\begin{Proposition}
Let $(X,d,\ll,\leq,\tau)$ be a finitely $\tau$-measurable and geodesic Lorentzian pre-length space. Then
\[
\mathrm{IniInj}_C^+(X)\geq\mathrm{UniqueInj}_C^+(X).
\]
If $X$ is regular, then also
\[
\mathrm{IniInj}_T^+(X)\geq\mathrm{UniqueInj}_T^+(X),
\qquad
\mathrm{IniInj}_N^+(X)\geq\mathrm{UniqueInj}_N^+(X).
\]
\end{Proposition}
\begin{proof}
Let $r\geq0$ have the uniqueness property in Definition~\ref{def:unique-injectivity-radii} for causally related pairs. If $q\in\mathrm{IniCut}_C^+(p)$ and $d(p,q)<r$, then there exist future causal cut points $q_i$ of $p$ converging to $q$. For sufficiently large $i$ we have $d(p,q_i)<r$, contradicting uniqueness of the maximizer from $p$ to $q_i$. Thus $d(p,q)\geq r$ for all $q\in\mathrm{IniCut}_C^+(p)$. Taking the infimum over $q$ and $p$, and then the supremum over all such $r$, gives $\mathrm{IniInj}_C^+(X)\geq\mathrm{UniqueInj}_C^+(X)$, also when one of the sets is empty or the supremum is infinite. The timelike and null cases follow by the same argument.
\end{proof}

The reverse inequality does not follow under these assumptions. In particular, an initial cut point lying earlier on a maximizer from $p$ to $q$ need not be at $d$-distance at most $d(p,q)$ from $p$.

\section{Applications}
\label{Section: applications}

{We give two applications of our notions of conjugate points. First, we prove a weak timelike Rauch comparison theorem under local upper curvature bounds and strong causality, in analogy with Shankar--Sormani~\cite[Thm.\ 10.1]{shankar2009conjugate}. Second, we give a consequence of the Lorentzian Cartan--Hadamard theorem of Er\"{o}s--Gieger~\cite{Eroes-Gieger25}. For the metric Cartan--Hadamard theorem, see Bridson--Haefliger~\cite[Thm.\ 4.1]{bridson2013metric}.}

\subsection{A weak timelike Rauch comparison theorem}
\label{subsection: Timelike Rauch}

We use the qualifier~\emph{weak\/} because the following result bounds the location of conjugate points. The classical Rauch theorem additionally gives a quantitative comparison of the growth of Jacobi fields. An analogue of that comparison for continuous timelike families is not established here.

\begin{Theorem}[Weak timelike Rauch comparison theorem]
\label{Theorem: Timelike Rauch}
Let $(X,d,\ll,\leq,\tau)$ be a chronological, strongly causal, regular and locally timelike maximizing Lorentzian pre-length space with timelike curvature bounded from above by $\kappa$. Let $\gamma\in\TGeo(X)$ have symmetrically conjugate endpoints. Then $L_\tau(\gamma)\geq D_\kappa$.
\end{Theorem}

\begin{proof}
Indirectly assume, that $L_\tau(\gamma)<D_\kappa$. Set $K=\min\{\kappa,0\}$. In our conventions, a timelike upper curvature bound by $\kappa$ implies one by every $K\leq\kappa$, cf.\ \cite[Lem.\ 3.27]{beran2022angles}, and $D_K=D_\kappa$. We may therefore work with $K\leq0$. By Lemma~\ref{lem:finite-tau-measurability}, $X$ is finitely $\tau$-measurable.

Every point of $\gamma$ has an open, causally convex, locally timelike maximizing $K$-comparison neighborhood $U$ with $\tau|_{U\times U}<D_K$ and $\tau$ continuous on $U\times U$. Indeed, chronology gives $\tau(x,x)=0$, so Definition~\ref{def:timelike-curvature-bounds}(i) allows us to shrink a comparison neighborhood around $x$ to a neighborhood $V$ with $\tau|_{V\times V}<D_K$. Intersect $V$ with a locally timelike maximizing neighborhood, and use strong causality to choose a smaller open, causally convex neighborhood $U$. Maximizers in the original comparison neighborhood remain in $U$ by causal convexity, and the other comparison properties pass to $U$ by restriction. Timelike maximizers with fixed endpoints in $U$ are unique up to increasing reparametrization by~\cite[Thm.\ 4.7]{BeranNapperRott2025}.

Choose distinct timelike geodesics $\gamma_n\neq\sigma_n$ from $p_n$ to $q_n$ as in the definition of symmetric conjugacy, with $\gamma_n,\sigma_n\to\gamma$ in $d_\Gamma$. As in the proof of Proposition~\ref{Proposition: TopologyofGeoX}, we may choose Lipschitz representatives converging uniformly to a fixed representative of $\gamma$ on $[0,1]$. 

Cover $\gamma([0,1])$ by finitely many neighborhoods $U$ as above. By Lebesgue's number lemma for the pullback cover of $[0,1]$, choose an even integer $N\geq4$ and a subdivision $0=t_0<t_1<\cdots<t_N=1$ sufficiently  fine so that every set
\[
\gamma\bigl([t_{\max\{0,i-2\}},t_{\min\{N,i+2\}}]\bigr)
\]
lies in one selected neighborhood $U_i$. Compactness and uniform convergence ensure that the corresponding sets of both approximating curves also lie in $U_i$ for all sufficiently large $n$. Openness of $\ll$ ensures, simultaneously for all indices, that any point on either curve at $t_i$ is timelike before any point on either curve at $t_{i+1}$. Length convergence also implies $\max\{L_\tau(\gamma_n),L_\tau(\sigma_n)\}<D_K$ 
for all sufficiently large $n$. Fix one such $n$ and set
\[
z_0=p_n,\qquad z_N=q_n,\qquad
z_i=\begin{cases}
\gamma_n(t_i),&i\text{ odd},\\
\sigma_n(t_i),&i\text{ even},
\end{cases}\quad 1\leq i\leq N-1,
\]
see Figure \ref{fig:rauch-illustration}. 
Choose a timelike maximizer $c_i$ from $z_i$ to $z_{i+1}$, taking $c_0$ and $c_{N-1}$ to be the corresponding portions of $\gamma_n$. Causal convexity ensures that same connector $c_i$ can be used in adjacent configurations.

For $0\leq i\leq N-2$, form the timelike triangle
\[
T_i=\Delta(z_i,z_{i+1},z_{i+2}),
\]
with short sides $c_i,c_{i+1}$ and long side the corresponding portion of $\sigma_n$ if $i$ is even, or of $\gamma_n$ if $i$ is odd. Each long side is maximizing by the choice of neighborhoods. 

\begin{figure}[t]
    \centering
    \includegraphics[width=0.92\linewidth,height=0.32\textheight,keepaspectratio]{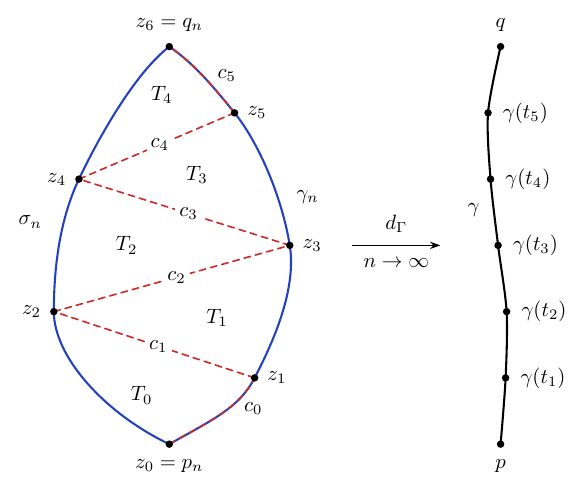}
    \caption{Configuration in $X$ for $N=6$. The vertices $z_i$ alternate between $\gamma_n$ (odd $i$) and $\sigma_n$ (even $i$). The dashed connectors $c_1,\ldots,c_4$, together with the boundary portions $c_0,c_5$, form the triangles $T_i=\Delta(z_i,z_{i+1},z_{i+2})$. The right-hand curve is the limit $\gamma$.}
    \label{fig:rauch-illustration}
\end{figure}

Set up the comparison triangles successively in $\lm{K}$, matching their common edges and giving consecutive triangles alternating orientations. The remaining vertices are placed on opposite sides of the common edge. The boundary paths $\bar\gamma_n,\bar\sigma_n$ in the comparison space follow 
\[
z_0,z_1,z_3,\ldots,z_{N-1},z_N
\quad\hbox{and}\quad z_0,z_2,z_4,\ldots,z_N,
\]
respectively, and have common initial and final model points $P,Q$. By additivity of length,
\[
L(\bar\gamma_n)=L_\tau(\gamma_n)<D_K,\qquad
L(\bar\sigma_n)=L_\tau(\sigma_n)<D_K.
\]

We now verify that the boundaries in the model space can only turn outwards. At $z_i$ ($2\leq i\leq N-2$) let $A_i,B_i,C_i$ be the nonnegative comparison angles at $z_i$ in $T_{i-2},T_{i-1},T_i$, respectively. The five vertices $z_{i-2},\ldots,z_{i+2}$ and their connecting sides lie in $U_i$, and $z_{i-2},z_i,z_{i+2}$ lie on one maximizing portion of an original geodesic. We claim that
\begin{equation}\label{eq:rauch-local-corner}
B_i\leq A_i+C_i.
\end{equation}
We set $x=z_i$ and call this boundary maximizer $\lambda$, parametrized by signed proper time, with $\lambda(0)=x$. Denote the incoming connector, traversed towards the past, by $\mu$, and the outgoing connector by $\nu$, both starting at $x$ and parametrized by positive proper time. 

Fix a sufficiently small $r>0$ so that $\lambda(-r)$ and $\lambda(r)$ lie on the boundary portion under consideration. For every fixed $\rho>0$, openness of $\ll$ gives, for all sufficiently small $\varepsilon>0$,
\[
\lambda(-r)\ll\mu(\rho\varepsilon)\ll x
\ll\nu(\varepsilon)\ll\lambda(r).
\]
Let $\alpha_\varepsilon,\beta_\varepsilon,\chi_\varepsilon$ be the $K$-comparison angles at $x$ between the successive pairs of legs in this configuration. Signed comparison-angle monotonicity, \cite[Prop.\ 3.9]{BKR24}, gives
\[
\alpha_\varepsilon\leq A_i,\qquad
\beta_\varepsilon\geq B_i,\qquad
\chi_\varepsilon\leq C_i.
\]
Set 
\[
\begin{aligned}
u_\varepsilon&=\tau(\lambda(-r),\mu(\rho\varepsilon)),\ 
v_\varepsilon&=\tau(\nu(\varepsilon),\lambda(r)),\ 
w_\varepsilon&=\tau(\mu(\rho\varepsilon),\nu(\varepsilon)).
\end{aligned}
\]
The Lorentzian cosine laws~\cite[Lem.\ 2.4]{beran2022angles}, with $r$ fixed, yield
\[
u_\varepsilon=r-\rho\varepsilon\cosh\alpha_\varepsilon+O(\varepsilon^2),
\qquad
v_\varepsilon=r-\varepsilon\cosh\chi_\varepsilon+O(\varepsilon^2).
\]
When $K<0$, setting $k=\sqrt{-K}$, the first expansion follows from
\[
\cos(ku_\varepsilon)=\cos(kr)\cos(k\rho\varepsilon)
+\sin(kr)\sin(k\rho\varepsilon)\cosh\alpha_\varepsilon.
\]
Continuity gives $u_\varepsilon\to r$, and $\sin(kr)>0$ because $0<r<D_K$. The case $K=0$ follows directly from the Minkowski cosine law. The remainders are uniform for $0\leq\alpha_\varepsilon\leq A_i$ and $0\leq\chi_\varepsilon\leq C_i$, with constants depending only on $K,r,\rho,A_i,C_i$.

Since $\lambda$ maximizes across $x$, the reverse triangle inequality gives
\[
2r\geq u_\varepsilon+w_\varepsilon+v_\varepsilon,
\qquad
w_\varepsilon\leq
\varepsilon(\rho\cosh A_i+\cosh C_i)+O(\varepsilon^2).
\]
In particular, $w_\varepsilon=O(\varepsilon)$. Expanding the mixed-orientation cosine law and solving for the angle therefore gives
\[
\cosh\beta_\varepsilon
=\frac{w_\varepsilon^2-(\rho^2+1)\varepsilon^2}
       {2\rho\varepsilon^2}+O(\varepsilon^2).
\]
Letting $\varepsilon\downarrow0$, we obtain, for every fixed $\rho>0$,
\[
\begin{aligned}
\cosh B_i
&\leq\frac{(\rho\cosh A_i+\cosh C_i)^2-\rho^2-1}{2\rho}\\
&=\cosh A_i\cosh C_i+\frac{\rho}{2}\sinh^2 A_i
+\frac{1}{2\rho}\sinh^2 C_i.
\end{aligned}
\]
At $z_1$ the incoming connector is already the incoming boundary segment, and at $z_{N-1}$ the outgoing connector is already the outgoing boundary segment. The same computation therefore applies with $A_1=0$ and $C_{N-1}=0$, respectively. 

To check the signs in the comparison configuration, write $\eta_b^-,\eta_c^-,\eta_c^+,\eta_b^+$ for the rapidities at the vertex of the incoming boundary tangent, incoming connector tangent, outgoing connector tangent and outgoing boundary tangent, all taken future-directed. Set $\epsilon_i=1$ on the right boundary and $\epsilon_i=-1$ on the left boundary. The alternating orientations give
\[
\eta_c^- -\eta_b^-=\epsilon_i A_i,\qquad
\eta_c^+ -\eta_c^-=-\epsilon_i B_i,\qquad
\eta_b^+ -\eta_c^+=\epsilon_i C_i.
\]
Adding these identities yields
\begin{equation}\label{eq:rauch-outward-turn}
\epsilon_i(\eta_b^+-\eta_b^-)=A_i+C_i-B_i\geq0.
\end{equation}
Thus, labelling the boundaries $R,L$ as right and left according to the first triangle, while every corner of $R$ turns rightward or is straight, every corner of $L$ turns leftward or is straight, and the initial tangent of $R$ lies to the right of, or equals, that of $L$.

Next, parametrize the two model space  boundaries by unit speed. We straighten them by isometries, starting at the last breakpoint and working backwards. At a breakpoint $V$, the entire remaining tail is then an unbroken geodesic. Replace this tail by its image under the orientation- and time-orientation-preserving isometry fixing $V$ whose derivative maps its initial unit tangent to the incoming unit tangent. This straightens the corner without changing the preceding portion. In particular, when the process reaches an earlier breakpoint, its two one-sided tangents are still the original ones. We repeat this procedure for both boundaries. The resulting unbroken geodesics retain their respective lengths $\ell_R,\ell_L<D_K$, their common initial point $P$ and their initial unit tangents $T_R,T_L$. Note, however, that their endpoints, denoted $\widehat Q_R,\widehat Q_L$, are not assumed to coincide.

To control these endpoints we define a function $f$ as follows: For $K=0$, let $E$ be the rightward unit spacelike vector orthogonal to the future unit bisector of $T_R,T_L$, and set $f(Z)=\langle E,Z-P\rangle$. For $K<0$, use the covering projection from $\lm{K}$ to the quadric $\langle Z,Z\rangle=1/K$ in $\mathbb R^3_2$, choose $E$ tangent at $P$ with the same properties, and set $f(Z)=\langle E,Z\rangle$. This function and the formulas below are pulled back to the universal cover. In both cases,
\[
f(P)=0,\qquad df_P(T_R)\geq0,\qquad df_P(T_L)\leq0.
\]
The gradient $G$ of $f$ is everywhere spacelike (and pointing to the right). In case $K=0$ it is $E$, while for $K<0$,
\[
G=E-KfZ,\qquad \langle G,G\rangle=1-Kf^2>0.
\]
Indeed, $G$ is continuous, never null, and rightward at $P$. Define
\[
S_K(a)=\begin{cases}
a,&K=0,\\[2pt]
\displaystyle\frac{\sin(\sqrt{-K}\,a)}{\sqrt{-K}},&K<0,
\end{cases}
\qquad C_K(a)=S_K'(a).
\]
Along each unit-speed timelike model geodesic, $f''=Kf$, hence
\begin{equation}\label{eq:rauch-model-solution}
f(\exp_V(aT))=C_K(a)f(V)+S_K(a)\,df_V(T).
\end{equation}
If $N_T$ is the rightward unit spacelike normal to a future unit timelike vector $T$, let  $T_\theta=\cosh\theta\,T+\sinh\theta\,N_T$. For $0<a<D_K$,
\[
\frac{d}{d\theta}f(\exp_V(aT_\theta))
=S_K(a)\langle G(V),\sinh\theta\,T+\cosh\theta\,N_T\rangle>0,
\]
because $S_K(a)>0$ and both vectors in the inner product are rightward spacelike.

At every backwards straightening step, the remaining tail has length in $(0,D_K)$. Straightening an outward corner on $R$ rotates its tail leftward, so the preceding inequality decreases the final value of $f$. Meanwhile, on $L$ it increases that value. The change is strict for a nonzero corner. As the original endpoint was $Q$ on both boundaries, we obtain 
\[
f(\widehat Q_R)\leq f(Q)\leq f(\widehat Q_L).
\]
On the other hand, the resulting straightened curves have initial data $(P,T_R)$ and $(P,T_L)$, so~\eqref{eq:rauch-model-solution} implies 
\[
f(\widehat Q_R)=S_K(\ell_R)\,df_P(T_R)\geq0,\qquad
f(\widehat Q_L)=S_K(\ell_L)\,df_P(T_L)\leq0
\]
and so 
\begin{equation}\label{eq:rauch-terminal-order}
0\leq f(\widehat Q_R)\leq f(Q)\leq f(\widehat Q_L)\leq0.
\end{equation}
Therefore all of the above inequalities are equalities. Strictness in the variation of the final value of $f$ forces every original corner to be straight, and positivity of $S_K(\ell_R),S_K(\ell_L)$ forces the initial opening to vanish. Thus every straightening step was trivial, and the original model space boundaries coincide to be the same unbroken timelike geodesic segment from $P$ to $Q$.

Every vertex in the model configuration lies on one of these boundaries. Since consecutive vertices are joined by future-directed timelike connectors, they lie on the common model geodesic in chronological order. Hence every comparison triangle is degenerate. Equality of corresponding side lengths now implies 
\begin{equation}\label{eq:rauch-collapse}
\tau(z_i,z_{i+2})=\tau(z_i,z_{i+1})+\tau(z_{i+1},z_{i+2}),
\qquad 0\leq i\leq N-2.
\end{equation}
Therefore $c_i*c_{i+1}$ is a timelike maximizer with the same endpoints as the original long side of $T_i$. But both lie in the same comparison neighborhood, so by local uniqueness they have to coincide. Replacing each boundary long side by this concatenation finally shows that both original curves coincide:
\[
\gamma_n=c_0*c_1*\cdots*c_{N-1}=\sigma_n
\]
up to increasing reparametrization. This contradicts $\gamma_n\neq\sigma_n$ and proves $L_\tau(\gamma)\geq D_K=D_\kappa$.
\end{proof}

\subsection{Nonexistence of conjugate points under nonpositive timelike curvature}
\label{subsection: Conjugate points in spaces with timelike curvature bounds}

We briefly recall the results of Er\"{o}s--Gieger~\cite{Eroes-Gieger25} used here. They showed that timelike curvature bounded above by $0$ implies local concavity of the time separation (Remark~2.2). In a strongly causal, regular and locally concave space, timelike geodesics contained in a comparison neighborhood are unique maximizers (Proposition~2.3). Under the hypotheses below, their Theorem~3.1 extends existence and uniqueness to a neighborhood of an entire timelike geodesic and gives uniform convergence as the endpoints approach those of the given geodesic. Our proof uses these facts to establish convergence in $d_\Gamma$ and the corresponding conclusions about conjugacy.

\begin{Theorem}
\label{Theorem: cartan hadamard adjacent}
Let $(X,d,\ll,\leq,\tau)$ be a Lorentzian pre-length space that is globally hyperbolic, strongly causal, locally causally closed, timelike path-connected,
and regular. Suppose that $X$ has curvature bounded above by $0$ in the sense of timelike triangle comparison and let $\gamma \in \TGeo(X)$. Then there exists a timelike family about $\gamma$ that is continuous at $\gamma$. This family is unique among timelike families about $\gamma$ that are continuous at $\gamma$. In particular, the endpoints of $\gamma$ are neither unreachable nor symmetrically conjugate along $\gamma$, and hence are not ultimate conjugate.
\end{Theorem}
\begin{proof} By~\cite[Rem.\ 2.2]{Eroes-Gieger25}, $X$ is locally concave. Therefore, \cite[Thm.\ 3.1]{Eroes-Gieger25} shows, in particular, that there exists an open neighborhood $U$ of $\gamma([0,1])$ and open neighborhoods $U_0$ of $\gamma(0)$ and $U_1$ of $\gamma(1)$ with $U_0 \ll U_1$ such that the following properties hold:
\begin{itemize}
\item[(i)] For any $(p,q)\in U_0\times U_1$ there exists a unique timelike geodesic $\gamma_{pq} \colon [0,1]\to U$ with $\gamma_{pq}(0)=p$ and $\gamma_{pq}(1)=q$.
\item[(ii)] If $p_k\in U_0$ with $p_k\to \gamma(0)$ and $q_k \in U_1$ with $q_k \to \gamma(1)$, then $\gamma_{p_kq_k}$ 
converges uniformly to $\gamma$. In particular, also $d_F(\gamma_{p_kq_k},\gamma) \to 0$.
\end{itemize}
Cover $\gamma([0,1])$ by comparison neighborhoods $V_1,\dots,V_N$ such that for a suitable partition $0=t_0<\dots<t_N=1$, 
$\gamma([t_i,t_{i+1}])\subseteq V_{i+1}$ for all $i=0,\dots,N-1$. By (ii), there exists some $k_0$ such that also $\gamma_{p_kq_k}([t_i,t_{i+1}])\subseteq V_{i+1}$ for all $k\ge k_0$ and all $i$. The proof of~\cite[Prop.\ 2.3]{Eroes-Gieger25} shows that 
then each $\gamma_{p_kq_k}|_{[t_i,t_{i+1}]}$ as well as $\gamma|_{[t_i,t_{i+1}]}$ are timelike maximizers. Since $\tau|_{V_{i+1}\times V_{i+1}}$ is
continuous, we obtain
\[
L_\tau(\gamma_{p_kq_k}|_{[t_i,t_{i+1}]}) = \tau(\gamma_{p_kq_k}(t_i),\gamma_{p_kq_k}(t_{i+1})) \to \tau(\gamma(t_i),\gamma(t_{i+1}))
= L_\tau (\gamma|_{[t_i,t_{i+1}]}).
\]
Consequently, 
\[
L_\tau(\gamma_{p_kq_k}) = \sum_{i=0}^{N-1} L_\tau(\gamma_{p_kq_k}|_{[t_i,t_{i+1}]})
\to \sum_{i=0}^{N-1} L_\tau(\gamma|_{[t_i,t_{i+1}]}) = L_\tau(\gamma),
\]
implying that in fact $\gamma_{p_kq_k} \to \gamma$ with respect to $d_\Gamma$. 

It follows that $F \colon U_0\times U_1 \to \TGeo(X)$, $(p,q) \mapsto \gamma_{pq}$ is a timelike family about $\gamma$ that is
continuous at $\gamma$. If $\tilde F \colon W_0\times W_1 \to \TGeo(X)$ is another timelike family about $\gamma$ that is continuous at $\gamma$, then $d_\Gamma$-continuity of $\tilde F$ at $\gamma$ implies that there exist open neighborhoods $\tilde W_i
\subseteq W_i\cap U_i$ ($i=0,1$) such that $\tilde F(p,q)([0,1])\subseteq U$ for all $(p,q)\in\tilde W_0\times\tilde W_1$. Then (i) implies that 
$F=\tilde F$ on $\tilde W_0\times \tilde W_1$.

The existence of $F$ shows that the endpoints of $\gamma$ are reachable. To exclude symmetric conjugacy, suppose that two sequences of distinct timelike geodesics with common endpoints converge to $\gamma$ in $d_\Gamma$. Since $\gamma([0,1])$ is compact and contained in the open set $U$, convergence in $d_F$ implies that all sufficiently late geodesics in both sequences have images contained in $U$. Their endpoints also lie in $U_0\times U_1$ for sufficiently large indices, so (i) forces the two geodesics with those endpoints to coincide, a contradiction.
\end{proof}

\section{Outlook}
\label{section: Outlook}

In this work, we have introduced and studied timelike conjugate points in Lorentzian (pre-)length spaces. We have compared various notions with one another and with their Jacobi field counterpart in the smooth setting. Moreover, we have given applications (Theorems~\ref{Theorem: Timelike Rauch} and~\ref{Theorem: cartan hadamard adjacent}) illustrating their use.

{Theorem~\ref{Theorem: Timelike Rauch} gives a bound on the location of symmetric conjugate points using local upper curvature bounds and strong causality. A quantitative comparison of the separation of nearby timelike geodesics, analogous to the growth estimates for Jacobi fields in the classical Rauch theorem, remains a further question.}

There are several avenues for future work on conjugate points in the synthetic Lorentzian setting. First, a study of conjugate points along null geodesics would be of great interest. Here, care should be taken not to disrupt the compatibility with the smooth case. 

Moreover, a link between conjugate points and singularity theorems should be established. In this context, there are two tasks: Obtaining conjugate points from curvature bounds, and then using conjugate points to obtain singularities. A version of the latter was investigated by Minguzzi in the context of closed cone structures~\cite{minguzzi2019causality}.

\begin{appendices}

\section{The Fr\'echet distance\texorpdfstring{\\}{ }between curves}
\label{Section: appendix frechet}

In this Appendix, we introduce the Fr\'echet distance between curves and prove some basic properties that were used in the main part of the paper. This notion indeed goes back to Fr\'echet and has found many applications (cf., e.g., \cite{Alt-Godau}), but we are not aware of a reference that rigorously establishes these fundamentals.

Let $(X,d)$ be a metric space. By a path in $X$ we mean an element $\gamma$ of the space of continuous maps $C([0,1],X)$ from $[0,1]$ into $X$ (we use the interval $[0,1]$ for convenience). Two paths $\gamma, \tilde{\gamma}$ are called equivalent if there exists a parametrization $\varphi \colon [0,1] \to [0,1]$ which is strictly increasing and surjective (hence continuous) such that $\gamma=\tilde{\gamma}\circ \varphi$. A curve is an equivalence class of paths under this equivalence relation, and we will
write $[\gamma]$ for the curve represented by the path $\gamma$. A path is called~\emph{nowhere constant\/} if it is not constant on any non-trivial sub-interval of $[0,1]$. This notion clearly descends to the space of curves as well. Given two curves $[\gamma],[\sigma]$, we define the~\emph{Fr\'echet distance\/} between them as follows:
\begin{equation}\label{eq:Frechet-distance}
    d_F([\gamma],[\sigma]) \coloneqq \inf_{\varphi,\psi} \max_{t \in [0,1]} d(\gamma(\varphi(t)),\sigma(\psi(t))),
\end{equation}
where $\varphi, \psi$ run over the set of all parametrizations as defined above. Obviously, we can equivalently write
\begin{equation*}
    d_F([\gamma],[\sigma]) \coloneqq \inf_{\varphi} \max_{t\in [0,1]} d(\gamma(\varphi(t)),\sigma(t)). 
\end{equation*}

We begin by establishing some properties of (spaces of) pairs of non-decreasing continuous maps:

\begin{Lemma}\label{Lem:parametrization} Let $\alpha \colon [0, 1] \to [0, 1]^2$ be a continuous, coordinate-wise non-decreasing map given by $\alpha(t) = (\varphi(t), \psi(t))$ with $\alpha(0) = (0, 0)$ and $\alpha(1) = (1, 1)$. Then there exists a coordinate-wise non-decreasing and $1$-Lipschitz continuous curve $\tilde{\alpha} = (\tilde{\varphi},\tilde{\psi}) \colon [0,2] \to [0,1]^2$ that has the same image as $\alpha$.
\end{Lemma}
\begin{proof} The function $s \colon [0,1]\to [0,2]$, $s(t) \coloneqq \varphi(t)+\psi(t)$ is continuous, non-decreasing, and surjective by the
intermediate value theorem. Let $\theta \colon [0,2] \to [0,1]$ be the generalized inverse of $s$,
\[
\theta(u) \coloneqq \inf \{ t\in [0,1] \mid s(t) \ge u \}.
\]
Since $s$ is surjective, $\theta$ is well-defined, and using continuity of $s$ one readily checks that $s(\theta(u)) = u$ for each $u\in [0,2]$. Now set, for $u\in [0,2]$:
\[
\tilde \alpha(u) \coloneqq \alpha(\theta(u)) = (\varphi(\theta(u)),\psi(\theta(u))) \eqqcolon (\tilde\varphi(u),\tilde\psi(u)).
\]
We first show that $\tilde\alpha([0,2]) = \alpha([0,1])$: Obviously, $\tilde\alpha([0,2]) \subseteq \alpha([0,1])$. Conversely, let
$p \coloneqq \alpha(t_0)$, $t_0\in [0,1]$ and $u_0 \coloneqq s(t_0)$. Setting $t_* \coloneqq \theta(s(t_0))$, we then have $\tilde\alpha(u_0) = \alpha(t_*)$. The definition of $\theta$, together with continuity and monotonicity of $s$ imply that $s(t_0)=s(t_*)$,
which we can also write as
\[
(\varphi(t_0)-\varphi(t_*)) + (\psi(t_0)-\psi(t_*)) = 0.
\]
Here, since $t_*\le t_0$ and since both $\varphi$ and $\psi$ are non-decreasing, both summands have to vanish, implying that $\alpha(t_0) = \alpha(t_*) =\tilde \alpha(u_0)$, establishing the claim.

Finally, to show that $\tilde \alpha$ is $1$-Lipschitz, let $0\le u<v\le 2$. Then using monotonicity of $\theta$, $\varphi$ and $\psi$,
hence of $\tilde \varphi$ and $\tilde\psi$,
\begin{align*}
\| \tilde\alpha(v) - \tilde \alpha(u)\|_1 &= |\tilde\varphi(v) - \tilde\varphi(u)| + |\tilde\psi(v)-\tilde\psi(u)|\\
&= (\tilde\varphi(v) + \tilde\psi(v)) - (\tilde\varphi(u) + \tilde\psi(u))\\ 
&= s(\theta(v)) - s(\theta(u)) = v-u.
\end{align*}
Consequently, for any $u,v\in [0,2]$ we have
\[
\| \tilde\alpha(v) - \tilde \alpha(u)\|_2 \le \| \tilde\alpha(v) - \tilde \alpha(u)\|_1 \le |v-u|.
\]
\end{proof}
\begin{Lemma}\label{Lem:parametrizations-Hausdorff-compact}
Let $N$ be the set of continuous, coordinate-wise non-decreasing maps $\alpha=(\varphi,\psi) \colon [0,1]\to [0,1]^2$ with $\alpha(0) = (0,0)$
and $\alpha(1)=(1,1)$. Then $\mathcal{N} \coloneqq \{\alpha([0,1]) \mid \alpha\in N\}$ is compact with respect to the Hausdorff distance $d_H$.
\end{Lemma}
\begin{proof} Denote by $L$ the set of all coordinate-wise non-decreasing $1$-Lipschitz maps $\eta \colon [0,2]\to [0,1]^2$ with $\eta(0)=(0,0)$ and $\eta(2)=(1,1)$. By the Theorem of Arzel\`a--Ascoli, $L$ is relatively compact in $C([0,2],[0,1]^2)$. 
Since it is clearly also closed with respect to uniform convergence, it is indeed compact in $C([0,2],[0,1]^2)$.

Denote by $\mathcal{K}([0,1]^2)$ the space of compact subsets of $[0,1]^2$, and equip it with the Hausdorff distance $d_H$. Consider the map 
\begin{align*}
    \iota \colon C([0,2],[0,1]^2) &\to \mathcal{K}([0,1]^2)\\
    \eta &\mapsto \eta([0,2]).
\end{align*}
Since $d_H(\iota(\eta), \iota(\zeta)) \le \sup_{u\in [0,2]} \|\eta(u) - \zeta(u)\|_2 $, it follows that $\iota$ is continuous (indeed, $1$-Lipschitz). Thus $\mathcal{L} = \iota(L)$ is compact with respect to $d_H$. Finally, by Lemma~\ref{Lem:parametrization}, $\mathcal{N} = \mathcal{L}$, concluding the proof.
\end{proof}

Based on these preparations, we can now show:
\begin{Proposition}\label{Prop:Frechet-distance} The Fr\'echet distance $d_F$ defines a metric on the space of nowhere constant curves in $(X,d)$.  
\end{Proposition}
\begin{proof} Since parametrizations form a group, the Fr\'echet distance is well defined. Symmetry of $d_F$ is clear and the triangle inequality follows readily from the definition. Also $d_F([\gamma],[\gamma])=0$ is obvious. 

The only non-trivial point therefore is to show that $d_F([\gamma],[\sigma])=0$ implies $[\gamma]=[\sigma]$, i.e., that $\gamma$ is a reparametrization of $\sigma$. From the definition of $d_F$ we obtain sequences of parametrizations $(\varphi_n,\psi_n)$ such that, for each $n\in\N$:
\begin{equation}\label{eq:phin-psin}
    \max_{t\in [0,1]} d(\gamma(\varphi_n(t)),\sigma(\psi_n(t))) < \frac{1}{n}.
\end{equation}

Set $\alpha_n \coloneqq (\varphi_n,\psi_n)$, and $A_n \coloneqq \alpha_n([0,1])$. Then by Lemma~\ref{Lem:parametrizations-Hausdorff-compact}, $(A_n)_{n\in \N}$
is a sequence in the $d_H$-compact set $\mathcal{N}$, hence possesses a $d_H$-convergent subsequence, which we may without loss of generality assume to be $(A_n)_{n\in \N}$ itself. Let $A\in \mathcal{N}$ be its limit. Then there is some $\alpha = (\varphi,\psi)\in N$ with $A=\alpha([0,1])$. Given any $t\in [0,1]$, there exists a sequence $(t_n)_{n\in \N}$ in $[0,1]$ such that 
\[
A_n\ni (\varphi_n(t_n),\psi_n(t_n)) \to (\varphi(t),\psi(t)) \qquad (n\to \infty).
\]
By \eqref{eq:phin-psin} we have $d(\gamma(\varphi_n(t_n)),\sigma(\psi_n(t_n))) < \frac{1}{n}$, so letting $n\to \infty$ we obtain
$d(\gamma(\varphi(t)),\sigma(\psi(t)))=0$, implying that $\gamma\circ \varphi = \sigma\circ \psi$. 

We are left with the task of constructing from $\varphi$ and $\psi$ new continuous surjective maps $\hat\varphi$, $\hat\psi \colon [0,1]\to [0,1]$ that are strictly monotonically increasing, hence qualify as parametrizations and satisfy $\gamma\circ \hat\varphi = \sigma\circ \hat\psi$. This we do
in several steps.

We first claim that $\varphi$ is constant on any non-trivial interval $[t_1,t_2]$ of $[0,1]$ if and only if $\psi$ is constant on $[t_1,t_2]$.
By symmetry, it suffices to show the `only if' part of the claim. Since $\varphi$ is constant on $[t_1,t_2]$, so is
$\gamma\circ \varphi = \sigma\circ \psi$.
Also, $\psi$ being continuous and non-decreasing, the image of $[t_1, t_2]$ under $\psi$ is the closed interval $J = [\psi(t_1), \psi(t_2)]$.
The path $\sigma$ is therefore constant on $J$. By our assumption on $\sigma$ this forces $J$ to be a single point, hence
$\psi(t_1) = \psi(t_2)$. Since $\psi$ is non-decreasing, $\psi$ is indeed constant on $[t_1, t_2]$.

Define $\tau \colon [0,1] \to [0,1]$ by
\[
\tau(t) = \frac{\varphi(t) + \psi(t)}{2}.
\]
Then $\tau$ is continuous, non-decreasing and surjective $[0,1]\to [0,1]$. 

Next we show that $\tau(t)$ is constant on an interval $[t_1, t_2]$ if and only if both $\vp(t)$ and $\ps(t)$ are constant on that interval. For the non-trivial direction of this claim, suppose that $\tau$ is constant on $[t_1, t_2]$. Then for any $t, t'$ with $t_1 \le t < t' \le t_2$, we have $\tau(t) = \tau(t')$, hence
\[
\frac{\vp(t) + \ps(t)}{2} = \frac{\vp(t') + \ps(t')}{2} \implies (\vp(t') - \vp(t)) + (\ps(t') - \ps(t)) = 0
\]
Since $\vp$ and $\ps$ are non-decreasing, we conclude that $\vp(t) = \vp(t')$ and $\ps(t) = \ps(t')$. Thus both $\vp$ and $\ps$ are constant on $[t_1, t_2]$.

For any $s \in [0,1]$, define $\hvp(s) \coloneqq \vp(t)$ and $\hps(s) \coloneqq \psi(t)$ for any $t \in \tau^{-1}(s)$. Since $\vp$, $\psi$ are constant on the fibers $\tau^{-1}(s)$ of $\tau$, we obtain well-defined functions 
$\hvp \colon [0,1] \to [0,1]$ and $\hps \colon [0,1] \to [0,1]$ such that
\[
\vp(t) = \hvp(\tau(t)) \quad \text{and} \quad \ps(t) = \hps(\tau(t)) \quad \text{for all } t \in [0,1].
\]
By construction, $\g\circ\vp = \sigma\circ \psi$, so also $\g(\hvp(\tau(t))) = \s(\hps(\tau(t)))$ for all $t\in [0,1]$.
Surjectivity of $\tau$ then gives $\g\circ \hvp = \sigma\circ \hps$. Also, since $\varphi$, $\psi$ are surjective, so are $\hvp$ and $\hps$.

Since $\tau$ is a continuous map from a compact space ($[0,1]$) to a Hausdorff space (again $[0,1]$), it is a quotient map. By the universal property of quotient maps, since $\vp$ and $\ps$ are continuous and constant on the fibers of $\tau$, the induced maps $\hvp$ and $\hps$ are also continuous.

It remains to show strict monotonicity, which we do for $\hat\varphi$. Let $0 \le s_1 < s_2 \le 1$. We must show $\hvp(s_1) < \hvp(s_2)$.
Since $\tau$ is non-decreasing and surjective from $[0,1]$ to $[0,1]$, we can find $t_1, t_2 \in [0,1]$ such that $\tau(t_1) = s_1$, $\tau(t_2) = s_2$, and $t_1 < t_2$. Then
\[
\hat\varphi(s_1) = \varphi(t_1) \le \varphi(t_2) = \hat\varphi(s_2).
\]
Let us indirectly assume that equality holds: $\hvp(s_1) = \hvp(s_2)$. Then $\vp(t_1) = \vp(t_2)$ and since $\vp$ is non-decreasing, it must be constant on the entire interval $[t_1, t_2]$. By what was shown above, then $\ps$ must also be constant on $[t_1, t_2]$.
But then also $\tau(t)$ must be constant on $[t_1, t_2]$. In particular, $\tau(t_1) = \tau(t_2)$, which implies $s_1 = s_2$, a contradiction.
This proves that $\hvp$ is strictly monotonically increasing, as claimed.
\end{proof}
\end{appendices}

\begin{ack}
We would like to thank Christina Sormani and Nathalie Tassotti for helpful discussions. Parts of the paper were improved through discussions with ChatGPT. 
\end{ack}

\begin{funding}
This research was funded in part by the Austrian Science Fund (FWF) [Grant DOI 10.55776/EFP6 and 10.55776/J4913]. For open access purposes, the authors have applied a CC BY public copyright license to any author accepted manuscript version arising from this submission.
\end{funding}

\end{document}